\documentclass{amsart} \usepackage{amscd,amsmath,amssymb}
 \usepackage{epsfig} \usepackage{xypic} \usepackage{graphicx}
 \newcommand{\Alt}{\mathfrak{A}} 
  \newcommand{\isom}{\xrightarrow{\sim}}

\DeclareMathOperator{\symd}{Symd}
\DeclareMathOperator{\charac}{char}
\DeclareMathOperator{\Spec}{Spec}
  \DeclareMathOperator{\SB}{\boldsymbol{\mathcal{SB}}}

 \DeclareMathAlphabet{\cat}{OT1}{cmss}{m}{sl}
 \DeclareMathOperator{\tens}{\otimes}
 \DeclareMathOperator{\CSA}{\cat{CSA}}

 \DeclareMathOperator{\Orth}{O}
  
 \DeclareMathOperator{\Gal}{Gal} \newcommand{\Set}{\cat{Set}}

\newcommand{\D}{\cat{D}}

  \newcommand{\DCov}{\cat{QCov}}
 \newcommand{\Proj}{\cat{Proj}}
 \newcommand{\Quad}{\cat{Quad}}
  \newcommand{\Et}{\cat{\acute Et}}
 \newcommand{\QCov}{\cat{QCov}} 
 \newcommand{\QEtex}{\cat{Q\acute Etex}} \newcommand{\OQCov}{\cat{OQCov}} 
 
\newcommand{\OQEtex}{\cat{OQ\acute Etex}}
\newcommand{\QCSA}{\cat{QCSA}}
\newcommand{\OQuad}{\cat{OQuad}}
\newcommand{\OQCSA}{\cat{OQCSA}}

  \newcommand{\Sym}{\mathfrak{S}}

 \DeclareMathOperator{\Iso}{Iso}
 \DeclareMathOperator{\funcX}{\mathbf{X}}
  \DeclareMathOperator{\funcY}{\mathbf{Y}}

 \DeclareMathOperator{\id}{Id}

 \DeclareMathOperator{\Hom}{Hom} \DeclareMathOperator{\End}{End}
 
  \DeclareMathOperator{\ind}{Ind}

  \DeclareMathOperator{\Skew}{Skew}
 \DeclareMathOperator{\SSym}{Sym} \DeclareMathOperator{\SSymd}{Symd}
 \DeclareMathOperator{\AAlt}{Alt}

 \DeclareMathOperator{\Trd}{Trd} \DeclareMathOperator{\Tr}{Tr}

  \DeclareMathOperator{\PGO}{PGO}
  
 \DeclareMathOperator{\PGL}{PGL}

 \DeclareMathOperator{\sym}{Sym}
 
 \newcommand{\n}{\boldsymbol{n}}

 \DeclareMathOperator{\funcM}{\mathbf{M}}
\DeclareMathOperator{\funcQ}{\mathbf{Q}}

 \DeclareMathOperator{\DSB}{\mathbb{SB}}

\newcommand{\joinrelshort}{\mathrel{\mkern-9mu}}
\newcommand{\shortlongrightarrow}{\relbar\joinrelshort\rightarrow}
\newcommand{\iso}{\mathrel{\mathop{\setbox0\hbox{$\mathsurround0pt
        \shortlongrightarrow$}\ht0=0.7\ht0\box0}\limits
    ^{\sim\mkern2mu}}}
    
     \newcommand{\ljoinrelshort}{\mathrel{\mkern-16mu}}
\newcommand{\lshortlongrightarrow}{\relbar\ljoinrelshort\leftarrow}
\newcommand{\liso}{\mathrel{\mathop{\setbox0\hbox{$\mathsurround0pt
        \lshortlongrightarrow$}\ht0=0.7\ht0\box0}\limits
    ^{\sim\mkern2mu}}}

\newcommand{\hcenter}[1]{{\leavevmode\setbox0\hbox{#1}\kern-.5\wd0\box0}}

\DeclareMathAlphabet{\cat}{OT1}{cmss}{m}{sly}

\theoremstyle{plain} %
\numberwithin{equation}{section}

\newtheorem{thm}[equation]{Theorem}
\newtheorem{principle}[equation]{Principle}
\newtheorem{prop}[equation]{Proposition}
\newtheorem{cor}[equation]{Corollary}
\newtheorem{lem}[equation]{Lemma}
\theoremstyle{definition} %
\numberwithin{equation}{section}
\newtheorem{example}[equation]{Example}

\newtheorem{remark}[equation]{Remark}

 \newtheorem{rem}{Remark}
\newcounter{eqalignnumcnt}

\newcommand{\OLDref}{} \let\OLDref\ref
\renewcommand{\ref}[1]{\textup{\OLDref{#1}}}

\newcommand{\nop}[3]{}

\dedicatory{\`A Jacques Tits, pour son $80^{\textit{\`eme}}$ anniversaire}

\title{Thin Severi--Brauer varieties} \author{Max-Albert Knus \and
  Jean-Pierre Tignol}
\address{Department Mathematik\\
  ETH Zentrum\\
  CH-8092 Z\"urich\\
  Switzerland} \email{knus@math.ethz.ch}
\address{Institut de Math\'ematique Pure et Appliqu\'ee\\
  Universit\'e catholique de Louvain\\
  B-1348 Louvain-la-Neuve\\
  Belgium} \email{jean-pierre.tignol@uclouvain.be} \thanks{The second
  author is supported in part by the F.R.S.--FNRS (Belgium)}

\begin{document}

\begin{abstract}
Severi--Brauer varieties are twisted forms of projective spaces
(in the sense of Galois cohomology) and are  associated in a functorial way
to central simple algebras. Similarly quadrics are related to algebras
with involution. Since thin projective spaces are finite sets, thin
Severi--Brauer varieties are finite sets endowed with a Galois action;
they are associated to \'etale algebras. Similarly, thin quadrics are
\'etale algebras with involution. We discuss embeddings of thin
Severi--Brauer varieties and thin quadrics in Severi--Brauer varieties
and quadrics as geometric analogues of embeddings of \'etale algebras
into central simple algebras (with or without involution), and
consider the geometric counterpart of the Clifford algebra construction.
\end{abstract}
\maketitle
\section{Introduction}

Paraphrasing Tits \cite{Tits74}, ``thin'' geometric objects are
characterised by the  
fact that their automorphism groups are the Weyl groups of the groups 
of automorphisms of the corresponding 
classical geometric objects.\footnote{In \cite{Tits57} Tits says that
such objects are defined over the ``field 
of characteristic 1''.}
Thus a thin $(n-1)$-dimensional projective 
space  is a finite set
of $n$ elements and a $2(n-1)$-dimensional thin quadric is a double
covering of a set of $n$ elements. 

Our starting point is to consider twisted forms
(in the sense of Galois cohomology) of projective spaces and quadrics.
Fixing a base field, a twist is given 
by a continuous action of the absolute Galois group of the
field. Twisted forms of projective spaces over a field $F$ occur as
Severi--Brauer varieties associated to 
central simple algebras over $F$, and twisted quadrics correspond to
central simple algebras with quadratic pair. Similarly, twisted forms
of finite sets are associated 
to algebraic objects, namely \'etale algebras, by a correspondence known as
Grothendieck's version of Galois theory. Double coverings correspond to
\'etale algebras with involution. 
Moreover the Grassmannian of maximal linear subspaces
of a split quadric of dimension $2(n-1)$ corresponds as a geometric
object to the Clifford algebra of a central algebra with quadratic
pair as algebraic object.  
We call this Grassmannian the Clifford set of the quadric.

The first aim of the paper is to study equivariant embeddings of
finite twisted objects into corresponding 
classical objects. For example a general embedding of a
finite set of $n$ elements
in a projective space of dimension $n-1$ corresponds to the embedding
of an \'etale algebra of rank $n$ in a central simple algebra
of degree $n$ (or to the embedding of a torus of rank~$n-1$ into a
twisted form of $\operatorname{SL_n}$, see
\cite[S2]{PrasadRapinchuk}). The second aim is to discuss in detail
the geometric side of the Clifford construction, and to relate it to
Clifford algebras for \'etale algebras with involution and for central
simple algebras with quadratic pairs.

In Section~\ref{sec:SB} we consider Severi--Brauer varieties, 
in Section~\ref{sec:involutions} quadrics, and in
Section~\ref{sec:Clifford}  the various types of Clifford
constructions, where
we introduce a special type of Severi--Brauer variety to describe
the geometry of a Clifford algebra.
Finally we sketch in Section~\ref{sec:Dynkin} a discussion of two
exceptional phenomena. First we describe consequences of
the fact that the Dynkin diagrams $A_3$ and $D_3$ are equal. This is 
related on the geometric side with the Klein quadric and on 
the algebraic side with the resolution of the equation of degree $4$.
Then we look at triality, that is, outer automorphisms of order 3 of $D_4$. 

One  origin of this paper is the observation of Tits~\cite[p. 287]{Tits57} 
that the two phenomena described above also occur over the 
``field of characteristic $1$''. Also, his short geometric proof of
the fact that two quaternion algebras whose tensor product is not
division share a common quadratic subfield \cite{Tits93} is a
beautiful illustration of the power of geometric insights into
algebraic questions. Another source of inspiration was~\cite{saltman},
where Saltman discusses the embedding of \'etale algebras with
involution into central simple algebras with involution.
Pairs consisting of a maximal \'etale subalgebra of a central simple algebra
(and corresponding pairs of algebras with involutions) are also considered in
\cite{chuard03} and \cite{PrasadRapinchuk}. Triality  in relation
with \'etale algebras is extensively discussed in \cite{KT09}.
We plan to come back to other aspects, in particular to unitary and
symplectic analogues, in \cite{KT10}.

\section{Severi--Brauer projective spaces}
\label{sec:SB}

\subsection{$\Gamma$-projective spaces and central simple algebras}
Throughout most of this work, $F$ is an arbitrary field. We denote by
$F_s$ a separable closure of $F$ and by $\Gamma$ the absolute Galois
group $\Gal(F_s/F)$, which is a profinite group.

A \emph{$\Gamma$-projective space $P$} of dimension $n-1$ over a field $F$ 
is a projective
space $\mathbb P^{n-1}(F_s)$ of dimension $n-1$ 
endowed with the discrete topology and a continuous action
of $\Gamma$ by collineations, denoted by $x \mapsto {}^\gamma x,\ x \in P,\
 \gamma \in \Gamma$. 
For every integer $k\in\{ 1,\ldots, n\}$,  the \emph{Grassmannian
$Gr_{k-1}(P)$} of $(k-1)$-dimensional linear varieties in $P$ carries an
action of $\Gamma$. 
We view the set:
 \[
\overline P =\{ x \in P \  | \  {}^\gamma x =x \ \text{for all $\gamma
  \in \Gamma$}\} \subset P 
\]
as a scheme over $F$.
The embedding $\overline  P \subset P$ induces an isomorphism
\[
 \beta \colon \overline  P\times_{\Spec F}  \Spec F_s \iso P
\]
of $\Gamma$-spaces. From now on we shall not distinguish between $P$ and the pair
$(\overline P,\beta)$. We have a similar construction for Grassmannians.

Continuous actions of $\Gamma$ are related with Galois cohomology, as we now
illustrate in the case of $\Gamma$-projective spaces.
There is a standard action of $\Gamma$ on  $\mathbb P^{n-1}(F_s)$ induced by 
the action of $\Gamma$ on $F_s$. We denote it by $x \mapsto {}^{\bar
  \gamma} x$.  
This standard action induces in turn a $\Gamma$-action on the group
$\PGL_n(F_s)$  of 
automorphisms of  the projective space $\mathbb P^{n-1}(F_s)$,
given by
$\varphi \mapsto \bar
\gamma( \varphi)= \bar\gamma \varphi {\bar\gamma}^{-1}, \ \varphi \in  \PGL_n(F_s) $.

A \emph{$1$-cocycle of $\Gamma$ with values in  $\PGL_n$} is a continuous map
$\Gamma \to \PGL_n(F_s)$, $\gamma \mapsto a_\gamma$,
 such that 
\[ a_{\gamma\delta} = a_\gamma\bar\gamma (a_\delta)  \quad \text{for all
$\gamma,\delta \in \Gamma$}.
\]
Two cocycles $a_\gamma,a'_\gamma$ are \emph{cohomologous or equivalent} if there exists 
$a \in \PGL_n(F_s)$ such that $a'_\gamma = a a_\gamma  {\bar\gamma}(a)^{-1})$.
We denote by $H^1(\Gamma,  \PGL_n)$ the set of equivalence classes of
$1$-cocycles. 

Let $P= (\overline P,\beta)$ be any $\Gamma$-projective space. The map 
$
a_\gamma = \beta\gamma\beta^{-1} \bar\gamma ^{-1}, \ 
\gamma \in \Gamma$,
is an element of $\PGL_n(F_s)$ and defines a $1$-cocycle. 
Moreover one shows that isomorphic
$\Gamma$-projective spaces lead to cohomologous cocycles. Let $\Proj_\Gamma^{n-1}$
be the groupoid\footnote{A groupoid is a category in which all 
  morphisms are isomorphisms.}  of $\Gamma$-projective spaces and let $\Iso\big(\Proj_\Gamma^{n-1} \big)$
be its set of isomorphism classes.
\begin{prop}
The map $(P,\beta)  \mapsto a_\gamma$ induces a bijection 

\[
\Iso\big(\Proj_\Gamma^{n-1} \big) \iso H^1(\Gamma,  \PGL_n).
\]
\end{prop}
\begin{proof}
We describe a map in the opposite direction and refer to
\cite[p. 161]{CorpsLocaux} 
for details. Given a $1$-cocycle 
$(a_\gamma)_{\gamma \in \Gamma}$, we define a $\Gamma$-action on
$P= \mathbb P^{n-1}(F_s)$ 
 as
\[
x \mapsto ^\gamma \!x = ^{a_\gamma\bar \gamma} \!x, \quad x \in P.
\qedhere
\]
\end{proof}

The  $\Gamma$-projective spaces we consider arise from the following
construction. Recall that an $F$-algebra $A$ is 
 \emph{central simple of degree $n$}
if there exists an isomorphism 
$\alpha\colon A \tens_{F} F_s \iso\End_{F_s}(V)$ for
a vector space $V$ of dimension $n$ over $F_s$. Since the group of
$F_s$-automorphisms of $\End_{F_s}(V)$ is isomorphic to  $\PGL_n(F_s)$,
a $1$-cocycle defing a $\Gamma$-projective space of dimension $n-1$ is
also a $1$-cocycle defining a central simple algebra of degree $n$.
 Thus denoting by $\CSA_F^n$
the groupoid of  central simple algebras of degree $n$, the sets
$\Iso\big(\Proj_\Gamma^{n-1} \big)$ and $\Iso\big(\CSA_F^n\big)$ are in bijection,
since they are in bijection with $H^1(\Gamma, \PGL_n)$.

There is a canonical way to associate a $\Gamma$-projective space to a 
central simple algebra. Let 
$\alpha\colon A \tens_{F} F_s \iso\End_{F_s}(V)$ as above.
Since the $n$-dimensional
right ideals of $\End_{F_s}(V)$ are of the form $\Hom_{F_s}(V,U)$ for
$U\subset V$ a $1$-dimensional subspace, any isomorphism
$A_s\simeq\End_{F_s}(V)$ defines a one-to-one correspondence between
$n$-dimensional right ideals of $A_s$ and $1$-dimensional subspaces of
$V$. Therefore, we may canonically associate to $A$ a projective space
$\SB(A)$ over $F_s$, whose $(k-1)$-dimensional linear varieties are
the $nk$-dimensional right ideals
of $A_s$. The action of the  group $\Gamma$ of $F_s$ over $F$
induces a continuous action on $A_s$, hence also on the right ideals
of $A_s$, and endows $\SB(A)$ with a $\Gamma$-projective space
structure. We call   $\SB(A)$ (or its associated scheme  $\overline{ \SB(A)}$ over $F$) 
the \emph{Severi--Brauer variety of $A$}. Its dimension is $n-1$, where
$n=\deg A$.
\medbreak

\begin{prop} \label{prop:SBequiv}
The rule $A \mapsto \SB(A)$ induces an anti-equivalence
\[
\CSA_F^n \equiv \Proj_\Gamma^{n-1}.
\]
\end{prop}

\begin{proof}
Applying \cite[Theorem~1, p.~93]{McL} it suffices to show that the
functor $\SB$ is fully faithfull and essentially surjective.
We refer to \cite{Jahnel} for details.
\end{proof}

\subsection{$\Gamma$-sets and \'etale algebras}

Finite sets with a continuous left action of $\Gamma$ (for the
discrete topology) are called (finite) $\Gamma$-sets. They form a
category $\Set_\Gamma$ whose morphisms are the $\Gamma$-equivariant
maps. Finite direct products and direct sums (= disjoint unions) are
defined in this category. We denote by $\lvert X\rvert$ the
cardinality of any finite set $X$. 

Following Tits \cite{Tits74} we define a \emph{thin
  $\Gamma$-projective space of dimension $n-1$}
as a $\Gamma$-set $X$ of $n$ elements. 
A \emph{thin $(k-1)$-dimensional linear subvariety} 
is a $\Gamma$-subset of $X$ with $k$ elements. For every integer $k, \ 1\leq k \leq n$, 
the $\Gamma$-set
\[
\begin{array}{lcl}
Gr_{k-1}(X) &= &\{\text{$(k-1)$-dimensional linear varieties contained in $X$}\}\\
&=& \{\text{$k$-element subsets of $X$} \}
\end{array}
\]
is a \emph{thin $\Gamma$-Grassmannian}. As for $\Gamma$-projective spaces
over $F$, we shall not distinguish between a finite $\Gamma$-set $X$ and its
scheme $\overline X$ over $F$  of fixed points of the $\Gamma$-action.
Like $\Gamma$-projective spaces over $F$, thin $\Gamma$-projective spaces
are related with  algebraic structures.
We recall that
a finite-dimensional commutative $F$-algebra $L$ is called
\emph{\'etale} (over $F$) if $L\otimes_FF_s$ is isomorphic to the
$F_s$-algebra $F_s^n=F_s\times\cdots\times F_s$ ($n$ factors) for some
$n\geq 1$.  \'Etale
$F$-algebras are the direct products of finite separable field
extensions of $F$. We refer to \cite[\S18.A]{KMRT} for various equivalent
characterizations of \'etale $F$-algebras.
These algebras (with $F$-algebra homomorphisms)
form a category $\Et_F$ in which finite direct products and finite
direct sums (= tensor products) are defined.
 

 

For any \'etale $F$-algebra $L$ of dimension $n$, the set of
$F$-algebra homomorphisms
\[
\funcX(L)=\Hom_\text{$F$-alg}(L, F_s)
\]
is a $\Gamma$-set of $n$ elements since $\Gamma$ acts on $F_s$.
Conversely, if $X$ is a $\Gamma$-set of $n$ elements, the $F$-algebra
$\funcM(X)$ of $\Gamma$-equivariant maps $X\to F_s$ is an \'etale
$F$-algebra of dimension~$n$,
\[
\funcM(X)=\{f\colon X\to F_s\mid \gamma\bigl(f(x)\bigr)= f({}^\gamma
x) \text{ for $\gamma\in\Gamma$, $x\in X$}\}.
\]
As first observed by Grothendieck, there are canonical isomorphisms
\[
\funcM\bigl(\funcX(L)\bigr)\cong L,\qquad
\funcX\bigl(\funcM(X)\bigr)\cong X,
\]
so that the functors $\funcM$ and $\funcX$ define an anti-equivalence
of categories\footnote{We let morphisms of $\Gamma$-sets act on the right 
  of the arguments (with the exponential notation) and use the usual 
  function notation for morphisms in the anti-equivalent category
  of \'etale algebras.}
\begin{equation}
  \label{eq:anteq}
  \Set_\Gamma\equiv\Et_F
\end{equation}
(see \cite[Proposition (4.3), p.~25]{deligne} or \cite[(18.4)]{KMRT}),
which we refer to as  the \emph{Grothendieck correspondence}.
Under this anti-equivalence, the cardinality of $\Gamma$-sets
corresponds to the dimension of \'etale $F$-algebras. The disjoint
union $\sqcup$ in $\Set_\Gamma$ corresponds to the product $\times $
in $\Et_F$ and the product $\times$ in $\Set_\Gamma$ to the tensor
product $\tens$ in $\Et_F$. For any integer $n\geq1$, we let $\Et^n_F$
denote the groupoid whose objects are $n$-dimensional
\'etale $F$-algebras and whose morphisms are $F$-algebra isomorphisms,
and $\Set^n_\Gamma$ the groupoid of $\Gamma$-sets with $n$ elements.
The anti-equivalence \eqref{eq:anteq} restricts to an anti-equivalence
$\Set_\Gamma^n \equiv\Et_F^n$.  \label{footnote1}
Isomorphism
classes are in bijection with  $H^1(\Gamma, \Sym_n)$,
where $\Gamma$ acts trivially on $\Sym_n$ (see \cite[29.9]{KMRT}).

The split \'etale algebra $F^n$
corresponds to the $\Gamma$-set $\boldsymbol{n}$ of $n$ elements with
trivial $\Gamma$-action.  \'Etale algebras of dimension~$2$ are also
called \emph{quadratic} \'etale algebras.

\subsection{Embeddings of $\Gamma$-sets in Severi--Brauer varieties}
 For $X$ a $\Gamma$-set of $n$ elements and a $\Gamma$-projective space
 $P$ of dimension $n-1$, we define a
\emph{general embedding}   
\[
\varepsilon\colon X\hookrightarrow P
\]
to be a
$\Gamma$-equivariant injective map whose image consists
of $n$ points in general position, i.e., not contained in a proper
subspace\footnote{Such embeddings are called 
\emph{frames} in \cite{Tits74}.}. The map $\varepsilon$ then induces a map 
\[
Gr_{k-1}\varepsilon\colon Gr_{k-1}(X) \hookrightarrow Gr_{k-1}(P)
\qquad\text{for $1\leq k\leq n$},
\]
which carries each $k$-element subset of $X$ to the linear variety
spanned by its image. 
In the rest of this subsection, we relate $F$-algebra embeddings of
$n$-dimensional \'etale $F$-algebras in a central simple $F$-algebra $A$ with general embeddings of
$\Gamma$-sets of $n$ elements in $\SB(A)$. We first discuss the case
where the \'etale $F$-algebra is split.

\begin{prop}
  \label{prop:splitembed}
  Let $L$ be a split \'etale $F$-algebra of dimension~$n$, and let $A$
  be a central simple $F$-algebra of degree~$n$. For every $F$-algebra
  embedding $\varepsilon\colon L\hookrightarrow A$, there is an
  $F$-algebra isomorphism $\varphi\colon A\to\End_FL$ such that
  $\varphi\circ\varepsilon\colon L\hookrightarrow\End_FL$ is the
  regular representation, which maps each $x\in L$ to multiplication
  by $x$.
\end{prop}

\begin{proof}
  Let $(e_i)_{1\leq i \leq n}$ be the collection of primitive idempotents of
  $L$. We have
  \[
  A=\varepsilon(e_1)A\oplus\cdots\oplus\varepsilon(e_n)A.
  \]
  Since the dimension of every right ideal of $A$ is a multiple of
  $n\,\ind A$, we must have
  \[
  \dim\varepsilon(e_i)A=n\text{ for $1\leq i\leq n$},\qquad
  \text{and}\quad\ind A=1.
  \]
  Identifying $A=\End_FV$ for some $F$-vector space $V$ of
  dimension~$n$, we may view $\varepsilon(e_1)$, \ldots,
  $\varepsilon(e_n)$ as projections on $1$-dimensional subspaces
  $V_1$, \ldots, $V_n$ such that $V=V_1\oplus\cdots\oplus V_n$. Any
  $F$-vector space isomorphism $V\to L$ that maps $V_i$ to $e_iL$ for
  all $1 \leq i \leq n$ determines an $F$-algebra isomorphism $\varphi\colon
  A\to\End_FL$ as required.
\end{proof}

Now, let $L$ be an arbitrary $n$-dimensional \'etale
$F$-algebra. Consider an $F$-algebra embedding into a central simple
$F$-algebra of degree~$n$:
\[
\varepsilon\colon L\hookrightarrow A.
\]
For every $\xi\in \funcX(L):=\Hom_{\text{$F$-alg}}(L,F_s)$, set
\begin{equation} \label{eq:epsstar}
\varepsilon_*(\xi)=\bigl\{
x\in A_s\mid \bigl(1\otimes\xi(\ell)\bigr)\cdot x =
\bigl(\varepsilon(\ell)\otimes1\bigr)\cdot x \text{ for all $\ell\in
  L$}\bigr\}.
\end{equation}

\begin{prop}
  \label{prop:embed}
  The right ideal $\varepsilon_*(\xi)$ is $n$-dimensional, and
  $\varepsilon_*$ is a general embedding
  \[
  \varepsilon_*\colon \funcX(L)\hookrightarrow\SB(A).
  \]
\end{prop}

\begin{proof}
  Let $L_s=L\otimes_FF_s$, a split \'etale $F_s$-algebra, and
  $\varepsilon_s=\varepsilon\otimes\id_{F_s}\colon L_s\hookrightarrow
  A_s$. By Proposition~\ref{prop:splitembed}, we may identify
  $A_s=\End_{F_s}L_s$ in such a way that $\varepsilon_s\colon
  L_s\hookrightarrow\End_{F_s}L_s$ is the regular representation. Let
  $(e_i)_{1\leq i\leq n}$ be the collection of primitive idempotents of
  $L_s$. Using the canonical correspondence
  $\Hom_{\text{$F$-alg}}(L,F_s)=\Hom_{\text{$F_s$-alg}}(L_s,F_s)$, we
  may set
  \[
  \funcX(L)=\{\xi_1,\ldots,\xi_n\}
  \]
  where $\xi_i\colon L_s\to F_s$ maps $e_i$ to $1$ and $e_j$ to $0$
  for $j\neq i$. For $f\in\End_{F_s}L_s$, the condition
  \[
  \bigl(1\otimes\xi_i(\ell)\bigr)\cdot f =
  \bigl(\varepsilon(\ell)\otimes 1\big)\circ f \qquad\text{for all
    $\ell\in L$}
  \]
  is equivalent to
  \[
  \bigl(1\otimes\xi_i(\ell)\bigr)\cdot f =
  \varepsilon_s(\ell)\circ f \qquad\text{for all
    $\ell\in L_s$.}
  \]
  For $\ell=e_1$, \ldots, $e_n$, this condition can be rewritten as
  \[
  e_j\cdot f(x)=0\quad\text{and}\quad e_i\cdot f(x)=f(x)
  \quad\text{for all $x\in L_s$.}
  \]
  Therefore, 
  \[
  \varepsilon_*(\xi_i)=\Hom_{F_s}(L_s,e_iL_s),
  \]
  which shows that $\dim\varepsilon_*(\xi_i)=n$. Moreover,
  \[
  \varepsilon_*(\xi_1)+\cdots+\varepsilon_*(\xi_n) =
  \Hom_{F_s}(L_s,e_1L_s+\cdots+e_nL_s)=\End_{F_s}L_s.
  \]
  Therefore $\varepsilon_*(\xi_1)$, \ldots,
  $\varepsilon_*(\xi_n)$ are in general position in the projective
  space $\SB(A)$.

  To complete the proof, we show $\varepsilon_*$ is
  $\Gamma$-equivariant. For $\gamma\in\Gamma$, we have
  \[
  \gamma\bigl(\varepsilon_*(\xi)\bigr) = \{x\in A_s\mid
  \bigl(1\otimes\xi(\ell)\bigr)\cdot\gamma^{-1}(x)=
  \bigl(\varepsilon(\ell)\otimes1\bigr)\cdot 
  \gamma^{-1}(x)\text{ for all $\ell\in L$}\}.
  \]
  By applying $\gamma$ to each side, we see that the condition
  \[
  \bigl(1\otimes\xi(\ell)\bigr)\cdot\gamma^{-1}(x)=
  \bigl(\varepsilon(\ell)\otimes1\bigr)\cdot 
  \gamma^{-1}(x)
  \]
  is equivalent to 
  \[
  \bigl(1\otimes{}^\gamma\xi(\ell)\bigr)\cdot x=
  \bigl(\varepsilon(\ell)\otimes1\bigr)\cdot x.
  \]
  Thus, $\gamma\bigl(\varepsilon_*(\xi)\bigr) = \varepsilon_*({}^\gamma\xi)$.
\end{proof}

Conversely, every general embedding $X\hookrightarrow\SB(A)$ of a
$\Gamma$-set $X$ of $n$ elements arises from an $F$-algebra embedding,
as we now show.

Given $n$-dimensional right ideals $I_1$, \ldots, $I_n\subset A_s$
such that $\{I_1,\ldots, I_n\}$ is stable under $\Gamma$ and
$I_1+\cdots+I_n=A_s$, let
\[
L_s=\{x\in A_s\mid xI_i\subset I_i\text{ for $1 \leq i\leq n$}\}.
\]
The set $L_s$ is an $F_s$-subalgebra of $A_s$ stable under the action of
$\Gamma$. Let $L=L_s^\Gamma$ be the fixed $F$-subalgebra.

\begin{prop}
  \label{prop:embedconv}
  The $F$-algebra $L$ is \'etale of dimension~$n$ and
  $\{I_1,\ldots,I_n\}$ is the image of the general embedding
  $\varepsilon_*\colon \funcX(L)\hookrightarrow\SB(A)$ induced by the
  inclusion $\varepsilon\colon L\hookrightarrow A$.
\end{prop}

\begin{proof}
  We may identify $A_s=M_n(F_s)$ in such a way that for $1 \leq i\leq n$ the
  ideal $I_i$ consists of matrices whose nonzero entries are on the
  $i$-th row. For $x=(x_{ij})_{1\leq i,j\leq n}$ we have $xI_j\subset I_j$ if
  and only if $x_{ij}=0$ for $i\neq j$. Therefore, $L_s$ is the
  algebra of diagonal matrices. Since $L_s=L\otimes_FF_s$, it follows
  that $L$ is an \'etale $F$-algebra of dimension~$n$. Under the
  canonical identification $\funcX(L)=\Hom_{\text{$F_s$-alg}}(L_s,F_s)$ we
  have $\funcX(L)=\{\xi_1,\ldots,\xi_n\}$ where $\xi_i\colon L_s\to F_s$
  yields the $i$-th diagonal entry. Matrix computation shows that
  $\varepsilon_*(\xi_i)=I_i$. 
\end{proof}

\begin{remark}
  The definition of the map $\varepsilon_*$ is inspired by the
  correspondence between maximal subfields of a division algebra and
  closed points on the corresponding Severi--Brauer variety set up by
  Merkurjev in \cite{Merk}.
\end{remark}

\section{Quadrics and Algebras with
  Involution} \label{sec:involutions}

\subsection{$\Gamma$-quadrics and quadratic pairs}
As in the preceding section, let $F$ be an arbitrary field and
$\Gamma$ the Galois group of some separable closure $F_s$ of $F$. We
call \emph{$\Gamma$-quadric} in a $(2n-1)$-dimensional
$\Gamma$-projective space $P$ over $F$ any quadric in $P$ defined by a
hyperbolic quadratic form and stable under the action of $\Gamma$. The
$\Gamma$-quadrics we consider in the sequel arise from quadratic pairs
on central simple algebras, which we now recall from~\cite[\S5]{KMRT}.
\medbreak\par
Let $B$ be an arbitrary $F$-algebra. An
\emph{involution}\footnote{We only consider $F$-linear involutions,
  which are sometimes called \emph{of the first kind}, when
  semi-linear involutions are allowed.} on $B$ is an
$F$-linear anti-automor\-phism of $B$ of order~$2$. The sets of
symmetric, skew-symmetric, symmetrized and alternating elements are the
following subspaces of $B$:
\[
\begin{array}{lcl}
\SSym(B,\sigma) &= &  \{x \in B \ \vert \ \sigma(x) = x\}\\
\Skew(B,\sigma) &= &  \{x \in B \ \vert \ \sigma(x) = -x\}\\
\SSymd(B,\sigma) &= &  \{x + \sigma(x) \ \vert \ x \in B \}\\
\AAlt(B,\sigma) &= &  \{x -\sigma(x) \ \vert \ x \in B \}.
\end{array}
\]
Now, let $A$ be a central simple $F$-algebra of degree~$2n$. A
\emph{quadratic pair} on $A$ is a pair $(\sigma,f)$ where $\sigma$ is
an involution on $A$ and $f\colon\SSym(A,\sigma)\to F$ is a linear
map, subject to the following conditions:
\begin{enumerate}
\item[(1)]
$\dim_F\SSym(A,\sigma) = n(2n+1)$ and
$ \Trd_A\big(\Skew(A,\sigma)\big) = \{0\}$;
\item[(2)]
$f\big(x +\sigma(x) \big)= \Trd_A(x)$ for all $x \in A$.
\end{enumerate}
If $A =\End_FV$, any quadratic pair $(\sigma,f)$ on $A$
is attached to a nonsingular 
quadratic form $q$ on $V$. The involution $\sigma$ is the involution
adjoint to the quadratic form $q$ and we refer to \cite[\S5B]{KMRT} for
the definition of $f$. 

For a central simple algebra $A$ over a field of characteristic not $2$ 
a quadratic pair  $(\sigma,f)$ is determined by the involution $\sigma$,
which has to be of orthogonal type (see \cite[\S5B]{KMRT}).
\medbreak\par
Any quadratic pair $(\sigma,f)$ on $A$ extends to a
quadratic pair $(\sigma_s,f_s)$ on $A_s=A\otimes_FF_s$. Since
$A_s$ is split, we may 
identify $A_s=\End_{F_s}(V)$ for some $2n$-dimensional $F_s$-vector
space $V$. Then $(\sigma_s,f_s)$ is adjoint to some hyperbolic
quadratic form $q$ 
on $V$, which is uniquely determined up to a scalar factor, see
\cite[(5.11)]{KMRT}. Recall from \cite[(6.5)]{KMRT} that a right ideal
$I\subset A_s$ is called \emph{isotropic for $(\sigma_s,f_s)$} when 
\[
\sigma_s(I)\cdot I=\{0\}\qquad\text{and}\qquad
f_s(I\cap\sym\sigma_s)=\{0\}.
\]
Isotropic ideals have the form $\Hom_{F_s}(V,W)$ where $W\subset V$ is
a totally isotropic subspace of $V$ for the form $q$, by
\cite[(6.6)]{KMRT}. Therefore, we may consider the quadric
$\funcQ(\sigma,f)$ in $\SB(A)$ whose points are the $2n$-dimensional right
ideals in $A_s$ that are isotropic for $(\sigma_s,f_s)$. The quadric is stable
under the action of the Galois group $\Gamma$ of $F_s/F$ and is given
by the equation $q=0$ in $\mathbb{P}(V)$ under the identification
$\SB(A)=\mathbb{P}(V)$. 

Let $\QCSA_F^{2n}$ be the groupoid of central simple algebras over $F$
of degree $2n$ 
with quadratic pairs and let $\Quad_\Gamma^{2(n-1)}$ be the groupoid
of $\Gamma$-quadrics of dimension $2(n-1)$ embedded in a
$\Gamma$-projective space of dimension $2n-1$. As in Proposition~\ref{prop:SBequiv},
we have:

\begin{prop}
\label{prop:QCSAQuad}
The rule $(A,\sigma,f) \mapsto
\funcQ(\sigma,f)$ 
induces an anti-equivalence of categories $\QCSA_F^{2n}\equiv
\Quad_\Gamma^{2(n-1)}$. 
\end{prop}

The automorphism group of a quadric of dimension 
$2(n-1)$  over $F_s$ or of a central simple algebra of degree $2n$
with a quadratic pair over 
$F_s$ is the projective orthogonal group $\PGO_{2n}(F_s)$ thus Galois
cohomology 
leads to  bijections of the sets $\Iso\big(\QCSA_F^{2n}  \big)$ and 
$\Iso\big( \Quad_\Gamma ^{2(n-1)} \big)$  with $H^1(\Gamma, \PGO_{2n})$.

\subsection{Thin $\Gamma$-quadrics and quadratic \'etale extensions}

A morphism of $\Gamma$-sets $Y_0\xleftarrow{\pi} Y$ is called a
\emph{$\Gamma$-covering} if the number of elements in each fiber
$y_0^{\pi^{-1}}\subset Y$ does not depend on $y_0\in Y_0$. This number is
called the \emph{degree} of the covering. Of particular importance in
the sequel are coverings of degree~$2$, 
which are also called \emph{double coverings}. Each such covering
$Y_0\xleftarrow{\pi}Y$ defines a canonical automorphism 
$Y\xleftarrow{\sigma} Y$ of order~$2$, which interchanges the elements
in each fiber of 
$\pi$. Clearly, this automorphism has no fixed points. We call
\emph{involution} of a $\Gamma$-set with an even number of elements
any automorphism of order~$2$ without fixed points. If
$Y$ is a $\Gamma$-set and $ Y\xleftarrow{\sigma} Y$ is an involution,
the set of orbits 
\[
Y/\sigma=\bigl\{\{y,y^\sigma\}\mid y\in Y\}
\]
is a $\Gamma$-set and the canonical map $(Y/\sigma)\leftarrow Y$ is a
double covering. Thus, the notions of double covering and involution
of $\Gamma$-sets are equivalent.
Paraphrasing Tits \cite[7.3]{Tits74} we call a double $\Gamma$-covering
$Y_0 \xleftarrow{\pi} Y$ with
$|Y_0| =n$, resp.\ a $\Gamma$-set with involution $(Y,\sigma)$ with
$\lvert Y\rvert=2n$,
a \emph{thin $2(n-1)$-dimensional $\Gamma$-quadric}.
Putting  $Y =\{a_1,b_1,a_2,b_2,\ldots,a_n,b_n\}$ where
$a_i^\pi=b_i^\pi$ for all~$i$, we view $Y_0$ as the set
of pairs $Y_0 = \bigl\{\{a_1,b_1\}, \{a_2,b_2\},\ldots,
\{a_n,b_n\}\bigr\}$. An \emph{isotropic 
linear subvariety  of dimension $k-1$} is a $k$-element subset of $Y$ not
containing any pair $\{a_i,b_i\}$. Thus, every isotropic linear
subvariety is contained in a maximal isotropic linear subvariety, which has
dimension~$n-1$. The maximal isotropic linear subvarieties may be viewed
as the sections of the map $\pi$; they form a $\Gamma$-set $C(Y/Y_0)$
(or $C(Y,\sigma)$) of $2^n$ elements, called the \emph{Clifford
  $\Gamma$-set} of $Y/Y_0$, which is studied in more detail in
  \S\ref{sec:Clifford}. 
\medbreak\par
\'Etale extensions are defined so as to correspond to
$\Gamma$-coverings under the anti-equivalence~\eqref{eq:anteq}:
\begin{prop}
  \label{prop:etex}
  For a homomorphism $L_0\xrightarrow{\varepsilon}L$ of \'etale
  $F$-algebras, the following conditions are equivalent:
  \begin{enumerate}
  \item[(a)] $\varepsilon$ endows $L$ with a structure of free
    $L_0$-module of rank~$d$;
  \item[(b)] the morphism of $\Gamma$-sets
    $\funcX(L_0)\xleftarrow{\funcX(\varepsilon)} \funcX(L)$ is a
    covering of degree $d$. 
  \end{enumerate}
\end{prop}

\begin{proof}
  Let $e_1$, \ldots, $e_r$ be the primitive idempotents of $L_0$, so
  that $L_0\cong M_1\times\cdots\times M_r$ with $M_i=e_iL_0$ a finite
  separable field extension of $F$. Then $\varepsilon(e_1)$, \ldots,
  $\varepsilon(e_r)$ are orthogonal idempotents of $L$ and, letting
  $L_i=\varepsilon(e_i)L$ we have $L\cong L_1\times\cdots\times L_r$
  where each $L_i$ is an $M_i$-vector space. Condition~(a) holds if and
  only if
  \[
  \dim_{M_1}L_1=\cdots=\dim_{M_r}L_r=d.
  \]
  To prove (a) and (b) are equivalent, we may extend scalars to $F_s$
  and therefore assume $F=F_s=M_i$ for all $i$. Then $(e_i)_{i=1}^r$
  is a base of $L_0$, and we may identify $\funcX(L_0)$ with the dual base
  $(e_i^*)_{i=1}^r$. The number of elements of $\funcX(L)$ in the
  fiber above $e_i^*$ is the dimension $\dim_{M_i}L_i$, so (a) and (b)
  are equivalent.
\end{proof}

When the equivalent conditions of Proposition~\ref{prop:etex} hold,
the map $\varepsilon$ is injective; identifying $L_0$ to a subalgebra of
$L$, we call $L$ an \emph{\'etale extension} of $L_0$ of
degree~$d$. \'Etale extensions of degree~$2$ are called
\emph{quadratic \'etale extensions}. They can also be defined by means
of automorphisms, as we now show.

Let $\sigma\colon L\to L$ be an automorphism of order~$2$ of an
\'etale $F$-algebra $L$, and let $L_0=\SSym(L,\sigma)\subset L$ denote the
$F$-subalgebra of fixed elements, which is necessarily \'etale. 

\begin{prop}\label{prop:invo1}
The following conditions are equivalent:
  \begin{enumerate}
  \item[(a)] the inclusion $L_0\hookrightarrow L$ is a
    quadratic \'etale extension of $F$-algebras;
  \item[(b)] the automorphism $\funcX(\sigma)$ is an involution on
    $\funcX(L)$. 
  \end{enumerate}
\end{prop}

\begin{proof}
  By the anti-equivalence~\eqref{eq:anteq} we have
  \[
  \funcX(L_0)=\funcX(L)/\funcX(\sigma).
  \]
  Therefore, (a) is equivalent by Proposition~\ref{prop:etex} to the
  condition that the canonical map
  $\funcX(L)/\funcX(\sigma)\leftarrow\funcX(L)$ is a
  $\Gamma$-covering of degree~$2$, hence also to~(b).
\end{proof}

When the equivalent conditions of Proposition~\ref{prop:invo1} hold,
the automorphism $\sigma$ is called an \emph{involution} of the \'etale
$F$-algebra $L$, whose dimension is then necessarily even. Thus, the
notions of quadratic \'etale extension and involution of \'etale
algebras are equivalent.

For $n\geq1$ we let
$\DCov_\Gamma^{n}$ denote the groupoid whose objects are thin
$\Gamma$-quadrics of dimension~$2(n-1)$, i.e., double coverings 
$Y_0 \xleftarrow{\pi} Y$ of $\Gamma$-sets with $|Y_0| =n$, and 
whose morphisms are isomorphisms of $\Gamma$-coverings.  Let
$\QEtex_F^{n}$ denote the groupoid of quadratic \'etale extensions  
$L/L_0$ of $F$-algebras with $\dim_FL_0 =n$.
\relax From~\eqref{eq:anteq} we
obtain an anti-equivalence of groupoids
\begin{equation}
\label{eq:QEtexQCov}
\QEtex^{n}_F\equiv\DCov^{n}_\Gamma.
\end{equation}
Isomorphism classes of these groupoids are in bijection with 
$H^1(\Gamma,\Sym_2^n\rtimes \Sym_n)$ (see \cite{knustignol}).
We denote by  $\funcY(L/L_0)$ or $\funcY(L,\sigma)$ the double
covering (i.e., the thin $2(n-1)$-dimensional quadric)
attached to the quadratic extension $L/L_0$:
\[
\funcY(L/L_0) = \funcX(L_0)\leftarrow\funcX(L),\qquad
\funcY(L,\sigma) = \funcX(L)/\funcX(\sigma)\leftarrow \funcX(L).
\]

The following observation on symmetric and alternating elements in an
\'etale algebra with involution will be used in the next subsection:

\begin{prop}
  \label{prop:symdetale}
  For any \'etale $F$-algebra of even dimension with
  involution $(L,\sigma)$, we have
  \[
  \sym(L,\sigma)=\symd(L,\sigma) \qquad\text{and}\qquad
  \Skew(L,\sigma)=\AAlt(L,\sigma). 
  \]
\end{prop}

\begin{proof}
  The proposition is clear if $\charac F\neq2$, for then
  \[\textstyle
  x=\frac{x}{2}+\sigma_L\bigl(\frac{x}{2}\bigr) \quad\text{and}\quad
  y=\frac{y}{2}-\sigma_L\bigl(\frac{y}{2}\bigr) \quad\text{for
    $x\in\sym(L,\sigma)$ and $y\in\Skew(L,\sigma)$.}
  \]
  For the rest of the proof, we may thus assume $\charac F=2$. The
  linear map $L\to\symd(L,\sigma)$ that carries $\ell\in L$ to
  $\ell+\sigma(\ell)$ fits into the exact sequence
  \[
  0\to\sym(L,\sigma)\to L\to\symd(L,\sigma)\to 0.
  \]
  Since $\dim_F\sym(L,\sigma)=\frac12\dim_FL$, it follows that
  $\dim\symd(L,\sigma)=\dim\sym(L,\sigma)$. It follows that the
  obvious inclusion $\symd(L,\sigma)\subseteq\sym(L,\sigma)$ is an
  equality, which proves the first equation. The second equation is
  the same as the first since $\charac F=2$.
\end{proof}

\subsection{Embeddings of quadrics and of algebras with involution}

Let $A$ be a central simple $F$-algebra of even degree~$2n$, and let
$(\sigma,f)$ be a quadratic pair on $A$. If $(L,\sigma_L)$ is an
\'etale $F$-algebra of dimension $2n$ with involution, we
define an \emph{embedding of algebras with involution}
\[
\varepsilon\colon(L,\sigma_L)\hookrightarrow(A,\sigma,f)
\]
to be an $F$-algebra embedding $\varepsilon\colon L\hookrightarrow A$
satisfying the following conditions:
\begin{enumerate}
\item[(a)]
$\varepsilon\circ\sigma_L=\sigma\circ\varepsilon$, and
\item[(b)]
$f(x)=\Trd_A(\varepsilon(\ell)x)$ for all $x\in\sym(\sigma)$ and all
$\ell\in L$ such that $\ell+\sigma_L(\ell)=1$.
\end{enumerate}
Note that Proposition~\ref{prop:symdetale} shows that there exists
$\ell\in L$ such that $\ell+\sigma_L(\ell)=1$, since
$1\in\sym(\sigma_L)$. If $\ell$, $\ell'\in L$ satisfy
\[
\ell+\sigma_L(\ell)=\ell'+\sigma_L(\ell')=1,
\]
then $\ell-\ell'\in\Skew(\sigma_L)$, and
Proposition~\ref{prop:symdetale} yields
$\ell-\ell'\in\AAlt(\sigma_L)$. If~(a) holds, then
$\varepsilon(\ell)-\varepsilon(\ell')\in\AAlt(\sigma)$, hence, by
\cite[(2.3)]{KMRT}, 
\[
\Trd_A\bigl(x(\varepsilon(\ell)-\varepsilon(\ell'))\bigr)=0
\qquad\text{for all $x\in\sym(\sigma)$.}
\]
Therefore, when~(a) holds, condition~(b) is equivalent to
\begin{enumerate}
\item[(b')]
there exists $\ell\in L$ such that $\ell+\sigma_L(\ell)=1$ and
$f(x)=\Trd_A(\varepsilon(\ell)x)$ for all $x\in\sym(\sigma)$.
\end{enumerate}
It can be further weakened to
\begin{enumerate}
\item[(b'')]
there exists $\ell\in L$ such that $f(x)=\Trd_A(\varepsilon(\ell)x)$
for all $x\in\sym(\sigma_L)$,
\end{enumerate}
because this condition on $\ell$ implies
$\varepsilon(\ell)+\sigma\varepsilon(\ell)=1$, see~\cite[(5.7)]{KMRT}.
Note that condition~(b') (hence also~(b'')) holds with $\ell=\frac12$
if $\charac F\neq2$. 

\begin{example}
\label{ex:embedinvo}
Let $(L,\sigma_L)$ be an arbitrary \'etale $F$-algebra with involution
of the first kind. Let $L_0=\sym(\sigma_L)\subset L$ be the fixed
subalgebra under $\sigma_L$, let $T_0\colon L_0\to F$ be the trace,
and let $t\colon L\to F$ be the quadratic form given by 
$t(x)=T_0\big(\sigma(x)x\big)$ for $x\in L$. Let $(\sigma_t,f_t)$ denote the
adjoint quadratic pair on $\End_FL$. The regular representation
is an embedding
\[
(L,\sigma_L)\hookrightarrow(\End_FL,\sigma_t,f_t).
\]
Indeed, condition~(a) holds because the polar form of $t$ is
\[
b_t(x,y)=T_{L_0/F}(\sigma_L(x)y+\sigma_L(y)x)\qquad \text{for $x$,
  $y\in L$,}
\]
and we have
\[
b_t(xy,z)=b_t(y,\sigma_L(x)z) \qquad\text{for $x$, $y$, $z\in L$.}
\]
Condition~(b) also is easily checked, see \cite[(5.12)]{KMRT}.
\end{example}

\begin{prop}
  \label{prop:splitinvembed}
  Use the same notation as in Example~\ref{ex:embedinvo}. Suppose $L$
  is split and let $(A,\sigma,f)$ be a central simple $F$-algebra with
  quadratic pair. For any embedding
  \[
  \varepsilon\colon(L,\sigma_L)\hookrightarrow(A,\sigma,f)
  \] 
  there is
  an isomorphism $\varphi\colon(A,\sigma,f)\to(\End_FL,\sigma_t,f_t)$
  such 
  that $\varphi\circ\varepsilon\colon(L,\sigma_L)\hookrightarrow
  (\End_FL,\sigma_t,f_t)$ is the regular representation.
\end{prop}

\begin{proof}
  Let $e_1$, \ldots, $e_n$, $e'_1$, \ldots, $e'_n$ be the collection
  of primitive idempotents of $L$, labelled in such a way that
  $e'_i=\sigma_L(e_i)$ for $i\in\n$. As observed in the proof of
  Proposition~\ref{prop:splitembed}, we may identify $A=\End_FV$ and
  we have a decomposition into $1$-dimensional subspaces
  \[
  V=V_1\oplus\cdots\oplus V_n\oplus V'_1\oplus\cdots\oplus V'_n
  \]
  such that $\varepsilon(e_i)$ (resp.\ $\varepsilon(e'_i)$) is the
  projection $V\to V_i$ (resp.\ $V\to V'_i$). Let $q\colon V\to F$ be
  a quadratic form such that $(\sigma,f)$ is adjoint to $q$, and let
  $b$ be the polar form of $q$. Since $\sigma$ is adjoint to $b$ and
  $\sigma\varepsilon(e_i)=\varepsilon\sigma_L(e_i)=\varepsilon(e'_i)$,
  we have for all $i$, $j\in \n$
  \[
  b(V_i,V_j)=b(\varepsilon(e_i)(V_i),V_j) =
  b(V_i,\varepsilon(e'_i)(V_j))=\{0\}.
  \]
  Similarly, $b(V'_i,V'_j)=\{0\}$ for all $i$, $j\in\n$, and
  $b(V_i,V'_j)=\{0\}$ if $i\neq j$. Since $b$ is nonsingular, it
  follows that $b(V_i,V'_i)=F$ for all $i\in\n$, hence we may find
  $v_i\in V_i$, $v'_i\in V'_i$ such that
  \[
  b(v_i,v'_i)=1.
  \]
  Let $S\in\End_FV$ be the map such that $S(v_i)=v'_i$ and
  $S(v'_i)=v_i$ for all $i\in\n$. It is readily verified that for
  $i\in\n$
  \[
  S\circ\varepsilon(e'_i)\colon x\mapsto v_ib(v_i,x)
  \quad\text{and}\quad
  S\circ\varepsilon(e_i)\colon x\mapsto v'_ib(v'_i,x)
  \quad\text{for all $x\in V$.}
  \]
  Therefore, by \cite[(5.11)]{KMRT},
  \[
  q(v_i)=f(S\circ\varepsilon(e'_i)) \qquad\text{and}\qquad
  q(v'_i)=f(S\circ\varepsilon(e_i)).
  \]
  If $\ell\in L$ is such that $f(x)=\Tr(\varepsilon(\ell)\circ x)$ for
  all $x\in\sym(\sigma)$, then
  \[
  f(S\circ\varepsilon(e'_i))=
  \Tr(\varepsilon(\ell)\circ S\circ\varepsilon(e'_i))=
  \Tr(S\circ\varepsilon(e'_i\ell)).
  \]
  None of the base vectors $v_1$, \ldots, $v'_n$ is mapped to a
  nonzero multiple of itself by $S\circ\varepsilon(e'_i\ell)$, hence
  $\Tr(S\circ\varepsilon(e'_i\ell))=0$ and therefore
  $q(v_i)=0$. Likewise, $q(v'_i)=0$ for $i\in\n$. Therefore, mapping
  $v_i\mapsto e_i$ and $v'_i\mapsto e'_i$ for $i\in\n$ defines an
  isometry $(V,q)\to(L,t)$, and the induced isomorphism of algebras
  with quadratic pairs $\varphi\colon(\End_FV,\sigma,f)\to
  (\End_FL,\sigma_t,f_t)$ is such that $\varphi\circ\varepsilon$ is
  the regular representation.
\end{proof}

We now turn to the geometric counterpart of the notion of embeddings
of algebras with involution. 
Let $Q\subset P$ be a $\Gamma$-quadric of dimension $2(n-1)$ in a
$\Gamma$-projective space of dimension~$2n-1$ over a field $F$. We let
$C(Q)$ denote the $\Gamma$-set of $(n-1)$-dimensional linear varieties
that lie in $Q$, and call this $\Gamma$-set the \emph{Clifford
  $\Gamma$-set of $Q$}. If $(Y,\sigma)$ is a thin $\Gamma$-quadric of
dimension~$2(n-1)$, we define an embedding of quadrics
\[
\varepsilon\colon(Y,\sigma)\hookrightarrow Q
\]
as a general embedding $\varepsilon\colon Y\hookrightarrow P$ that
carries isotropic linear subvarieties of $(Y,\sigma)$ to isotropic
linear subvarieties of $P$, i.e., to linear subvarieties of $Q$. In
particular, the image of $Y$ lies on $Q$. Since
every isotropic linear subvariety is contained in a maximal isotropic
linear subvariety, it suffices to require
that the induced map $Gr_{n-1}\varepsilon\colon
Gr_{n-1}Y\hookrightarrow Gr_{n-1}(P)$ 
carry the set of maximal isotropic linear subvarieties
$C(Y,\sigma)$ to $C(Q)$. We let $C(\varepsilon)$ denote the
$\Gamma$-equivariant map induced by $Gr_{n-1}\varepsilon$,
\[
C(\varepsilon)\colon C(Y,\sigma)\hookrightarrow C(Q).
\]

Let $(L,\sigma_L)$ be an arbitrary \'etale $F$-algebra of
dimension $2n$ with involution of the first kind, and let
\[
\varepsilon\colon (L,\sigma_L)\hookrightarrow (A,\sigma,f)
\]
be an embedding in a central simple $F$-algebra of degree~$2n$ with
quadratic pair. Since $\varepsilon\colon L\hookrightarrow A$ is an
$F$-algebra embedding, we may consider the general embedding of
Proposition~\ref{prop:embed}:
\[
\varepsilon_*\colon \funcX(L)\hookrightarrow \SB(A).
\]
Let $\funcY(L,\sigma_L)$ be the thin $\Gamma$-quadric 
associated with $(L,\sigma_L)$
and $Q(A,\sigma,f)$ the $\Gamma$-quadric associated with
$(A,\sigma,f)$.

\begin{prop}
  \label{prop:embedquad}
  The map $\varepsilon_*$ defines an embedding of quadrics
  \[
  \varepsilon_*\colon \funcY(L,\sigma_L) \hookrightarrow
  \funcQ(\sigma,f).
  \]
\end{prop}

\begin{proof}
  As usual, let $A_s=A\otimes_FF_s$, and let $(\sigma_s,f_s)$ be the
  quadratic pair on $A_s$ extending $(\sigma,f)$.
  We have to show that for every $\omega\in
  C\bigl(\funcY(L,\sigma_L)\bigr)$ we have
  $Gr_{n-1}\varepsilon_*(\omega)\in C\bigl(Q(\sigma,f)\bigr)$.
  Let $\xi$, $\eta\in \funcX(L)$. For $x\in\varepsilon_*(\xi)$,
  $y\in\varepsilon_*(\eta)$ and $\ell\in L$ we have
  \[
  \bigl(1\otimes\xi(\ell)\bigr)\cdot x =
  (\varepsilon(\ell)\otimes1)\cdot x \quad\text{and}\quad
  \bigl(1\otimes\eta(\ell)\bigr)\cdot
  y=(\varepsilon(\ell)\otimes1)\cdot y,
  \]
  hence also
  \[
  \sigma_s(x)\cdot(\varepsilon(\ell)\otimes1) = \sigma_s(x)\cdot
  \bigl(1\otimes\xi\sigma_L(\ell)\bigr),
  \]
  and
  \[
  \sigma_s(x)\cdot y\cdot\bigl(1\otimes\eta(\ell)\bigr) =
  \sigma_s(x)\cdot(\varepsilon(\ell)\otimes1)\cdot y =
  \sigma_s(x)\cdot y\cdot\bigl(1\otimes\xi\sigma_L(\ell)\bigr).
  \]
  If $\eta\neq\xi\sigma_L$, we can find $\ell\in L$ such that
  $\eta(\ell)\neq\xi\sigma_L(\ell)$, hence the last equation shows
  that $\sigma_s(x)\cdot y=0$. Therefore, 
  \[
  \sigma_s\bigl(\varepsilon_*(\xi)\bigr)\cdot
  \varepsilon_*(\eta)=\{0\}, \qquad\text{unless $\eta=\xi\sigma_L$.}
  \]
  In particular, for every $\omega\in C\bigl(\funcY(L,\sigma_L)\bigr)$,
  \begin{equation}
    \label{eq:isot}
    \sigma_s\Bigl(\bigoplus_{\xi\in\omega}\varepsilon_*(\xi)\Bigr)
    \cdot \bigoplus_{\xi\in\omega}\varepsilon_*(\xi) = \{0\},\quad
    \text{i.e.,}\quad
    \sigma_s\bigl(Gr_{n-1}\varepsilon_*(\omega)\bigr)\cdot
    Gr_{n-1}\varepsilon_*(\omega)=\{0\}.  
  \end{equation}

  Now, let $\ell\in L$ be such that
  $f(x)=\Trd(\varepsilon(\ell)x)$ for all $x\in\sym(\sigma)$. For
  $x\in Gr_{n-1}\varepsilon_*(\omega)\cap\sym(\sigma_s)$ we have
  \begin{equation}
    \label{eq:f}
    f_s(x)=\Tr\bigl((\varepsilon(\ell)\otimes1)\cdot x\bigr)
    \quad\text{since $x\in\sym(\sigma_s)$,}
  \end{equation}
  and
  \[
  (\varepsilon(\ell)\otimes1)\cdot x\in
  Gr_{n-1}\varepsilon_*(\omega) \quad\text{since
    $x\in Gr_{n-1}\varepsilon_*(\omega)$.}
  \]
  On the other hand, we have $x=\sigma_s(x)\in
  \sigma_s\bigl(Gr_{n-1}\varepsilon_*(\omega)\bigr)$, hence by
  \eqref{eq:isot},
  \[
  x\cdot\big(\varepsilon(\ell)\otimes1\big)\cdot x = 0.
  \]
  It follows that $\bigl((\varepsilon(\ell)\otimes1)\cdot
  x\bigr)^2=0$, hence $\Tr\bigl((\varepsilon(\ell)\otimes1)\cdot
  x\bigr)=0$ and \eqref{eq:f} shows that $f_s(x)=0$. Thus,
  $Gr_{n-1}\varepsilon_*(\omega)$ is isotropic for all $\omega\in
  C\bigl(\funcY(L,\sigma_L)\bigr)$, and the proposition is proved.
\end{proof}

Conversely, every embedding
\[
\varepsilon\colon(Y,\varphi)\hookrightarrow Q(\sigma,f)
\]
of a $\Gamma$-set $Y$ of $2n$ elements with an involution $\varphi$ of the first kind in the quadric associated with a quadratic pair on a central simple algebra $A$ of degree~$2n$ arises from the embedding of the \'etale $F$-algebra with involution $\funcM(Y,\varphi)$ into $(A,\sigma,f)$, as we now show.

Let $\varepsilon(Y)=\{I_1,\ldots, I_n, I'_1, \ldots, I'_n\}$, where the isotropic $2n$-dimensional right ideals $I_1$, \ldots, $I'_n$ of $A_s$ are indexed in such a way that $I'_i=\varepsilon(\xi^\varphi)$ if $\varepsilon(\xi)=I_i$. As in \S\ref{sec:SB}, let
\[
L_s=\{x\in A_s\mid x\cdot I_i\subset I_i\text{ and } x\cdot I'_i\subset I'_i\text{ for all $i\in\n$}\}
\]
and $L=L_s^\Gamma$. Proposition~\ref{prop:embedconv} shows that we may identify $X$ with $\funcX(L)$ in such a way that $\varepsilon$ is the general embedding $\funcX(L)\hookrightarrow\SB(A)$ induced by the inclusion $L\hookrightarrow A$. Let $\sigma_L$ be the involution on $L$ such that $\funcX(\sigma_L)=\varphi$.

\begin{prop}
  \label{prop:embedconvquad}
  The inclusion $L\hookrightarrow A$ is an embedding
  $(L,\sigma_L)\hookrightarrow(A,\sigma,f)$. 
\end{prop}

\begin{proof}
  The primitive idempotents $e_1$, \ldots, $e_n$, $e'_1$, \ldots,
  $e'_n$ of $L_s$ are in one-to-one correspondence with $I_1$, \ldots,
  $I_n$, $I'_1$, \ldots, $I'_n$ in such a way that $I_i=e_i\cdot A_s$
  and $I'_i=e'_i\cdot A_s$ for $i\in\n$. The indexing of $I_1$,
  \ldots, $I'_n$ is such that $\sigma_{L_s}$ interchanges $e_i$ and
  $e'_i$ for $i\in\n$. To check that $\sigma$ restricts to $\sigma_L$,
  observe that $I'_i$ is the only element among $\{I_1, \ldots,
  I'_n\}$ that does not belong to any $\omega^{C(\varepsilon)}$ for any
  $\omega\in C(X,*)$ containing $i$. Therefore,
  \[
  \sigma_s(I_i)\cdot I_j = \sigma_s(I_i)\cdot I'_k=\{0\}
  \qquad\text{for all $j$, $k\in\n$ with $k\neq i$,}
  \]
  hence for all $i\in\n$
  \[
  \sigma_s(e_i)\cdot e_j=\sigma_s(e_i)\cdot e'_k=0 \qquad\text{for all
    $j$, $k\in\n$ with $k\neq i$.}
  \]
  Similarly, for all $i\in\n$,
  \[
  \sigma_s(e'_i)\cdot e_j=\sigma_s(e'_i)\cdot e'_k=0 \qquad\text{for
    all $j$, $k\in\n$ with $j\neq i$.}
  \]
  Since $e_1+\cdots+e_n+e'_1+\cdots+e'_n=1$, it follows that for all
  $i\in\n$ we have
  \[
  e'_i=\Bigl(\sum_{j\in\n} \sigma_s(e_j)+\sigma_s(e'_j)\Bigr)\cdot
  e'_i = \sigma_s(e_i)\cdot e'_i
  \]
  and, on the other hand,
  \[
  \sigma_s(e_i)\cdot(1-e'_i)= \sigma_s(e_i) \cdot\Bigl(\sum_{j\in\n}
  e_j + \sum_{k\in\n\smallsetminus\{i\}} e'_k\Bigr) =0,
  \]
  hence $\sigma_s(e_i)=\sigma_s(e_i)\cdot e'_i$. Therefore,
  $\sigma_s(e_i)=e'_i$ for $i\in\n$.

  To complete the proof that $(L,\sigma_L)\hookrightarrow
  (A,\sigma,f)$, it remains to show that for all $x\in\sym(\sigma)$
  and all $\ell\in L$ such that $\ell+\sigma_L(\ell)=1$ we have
  $f(x)=\Trd_{A}(\ell x)$. To show that this equation holds for all
  $x\in\sym(\sigma_s)$ and all $\ell\in L_s$ such that
  $\ell+\sigma_s(\ell)=1$, it actually suffices to prove it for a
  single $\ell$, as was observed in the equivalence of conditions~(b)
  and (b') defining an embedding. Let $\ell=\sum_{i\in\n}e_i$. Since
  $\ell+\sigma_s(\ell)=1$ we have for all $y\in A_s$
  \[
  \Trd_{A_s}\bigl(\ell\cdot(y+\sigma_s(y))\bigr) =
  \Trd_{A_s}\bigl((\ell+\sigma_s(\ell))\cdot y\bigr) = \Trd_{A_s}(y).
  \]
  Since $f_s\bigl(y+\sigma_s(y)\bigr)=\Trd_{A_s}(y)$, we have
  \begin{equation}
    \label{eq:symd}
    f_s(x)=\Trd_{A_s}(\ell x) \qquad\text{for all
      $x\in\symd(\sigma_s)$.} 
  \end{equation}
  Let $I=\ell\cdot A_s= I_1+\cdots+I_n$ and $I'=(1-\ell)\cdot A_s =
  I'_1+\cdots+I'_n$. These right ideals are isotropic since
  $\{I_1,\ldots, I_n\}$, $\{I'_1,\ldots, I'_n\}\in
  C(\varepsilon(X,*))$. For $x\in I\cap\sym(\sigma_s)$ we have
  $f_s(x)=0$ and, on the other hand, $\ell x\in I$ and $x\ell x\in
  \sigma_s(I)\cdot I=\{0\}$, hence $(\ell x)^2=0$. Therefore,
  \begin{equation}
    \label{eq:I}
    f_s(x)=\Trd_{A_s}(\ell x) \qquad\text{for all
      $x\in I\cap\sym(\sigma_s)$.}
  \end{equation}
  Similarly, for $x\in I'\cap\sym(\sigma_s)$ we have $\ell x=0$, hence
  \begin{equation}
    \label{eq:I1}
    f_s(x)=\Trd_{A_s}(\ell x) \qquad\text{for all
      $x\in I'\cap\sym(\sigma_s)$.}
  \end{equation}
  Now, for $x\in\sym(\sigma_s)$ we have
  \[
  x=\ell x(1-\ell) + (1-\ell)x\ell + \ell x\ell + (1-\ell)x(1-\ell).
  \]
  Since $\ell\in I$, $1-\ell\in I'$, and $\sigma_s(\ell)=1-\ell$, we
  have
  \[
  \ell x(1-\ell)\in I\cap\sym(\sigma_s),\qquad (1-\ell)x\ell \in
  I'\cap\sym(\sigma_s),
  \]
  and
  \[
  \ell x\ell +(1-\ell)x(1-\ell) = \ell x \ell +\sigma_s(\ell x\ell)
  \in\symd(\sigma_s).
  \]
  Therefore, 
  \[
  f_s(x)=f_s\bigl(\ell x(1-\ell)\bigr) + f_s\bigl((1-\ell)x\ell\bigr)
  + f_s\bigl(\ell x\ell + (1-\ell)x(1-\ell)\bigr),
  \]
  and it follows from~\eqref{eq:symd}, \eqref{eq:I}, and \eqref{eq:I1}
  that
  \[
  f_s(x)=\Trd_{A_s}(\ell x).
  \]
\end{proof}

\section{Clifford algebras and $\Gamma$-sets}
\label{sec:Clifford}

As in the previous sections, we let $F$ denote an arbitrary field and
$\Gamma$ the Galois group of a separable closure $F_s$ of $F$.

\subsection{Clifford constructions}
Recall from \S\ref{sec:involutions} that for any $2(n-1)$-dimensional thin
$\Gamma$-quadric $Y/Y_0$ we let $C(Y/Y_0)$ denote the $\Gamma$-set of
maximal linear isotropic subvarieties, which can be identified with
the $2^n$-element set of sections of the covering $Y_0\leftarrow
Y$. We call $C(Y/Y_0)$ the \emph{Clifford $\Gamma$-set} of
$Y/Y_0$. The canonical involution $\sigma$ of $Y/Y_0$ induces an
involution $\underline{\sigma}$ on $C(Y/Y_0)$, which maps every
section to its complement in $Y$. If $n=1$, then $C(Y/Y_0)=Y$ and
$\underline{\sigma}=\sigma$. The following proposition is clear:

\begin{prop}\label{prop:isounion}
If $Y/Y_0$ and $Z/Z_0$ are thin $\Gamma$-quadrics, then $(Y\sqcup
Z)/(Y_0\sqcup Z_0)$ is a thin $\Gamma$-quadric and there is a
canonical isomorphism 
\[
C(Y/Y_0) \times C(Z/Z_0) \iso C(Y \sqcup Z
/Y_0\sqcup Z_0)
\]
which maps any pair of isotropic linear varieties $(\omega,\omega')$
to $\omega\sqcup\omega'$.
\end{prop}

We now turn to the analogue of this construction for quadratic \'etale
extensions. Let $\varepsilon\colon L_0\hookrightarrow L$ be a
quadratic \'etale extension of \'etale $F$-algebras, with
$\dim_FL_0=n$. In view of the Grothendieck correspondence there exists
an \'etale algebra $C(L/L_0)$ of dimension $2^n$, with involution
$\underline\sigma$, such that 
$C\funcY(L/L_0) = \funcY\big(C(L/L_0), \underline\sigma\big)$. 
This \'etale algebra is constructed explicitly in \cite{knustignol},
where it is designated by $\Omega(L/L_0)$. We briefly recall the
construction: the algebra $L_0^{\otimes n}$ contains an idempotent
$e_n^{L_0}$ that corresponds under the composition of canonical
isomorphisms
\[
L_0^{\otimes n} \iso \funcM\bigl(\funcX(L_0^{\otimes n})\bigr) \iso
\funcM(\funcX(L_0)^n)
\]
to the characteristic function $\chi$ of the complement of the ``fat
diagonal'' of $\funcX(L_0)^n$:
\[
\chi(\xi_1,\ldots,\xi_n)=
\begin{cases}
  1&\text{if $\xi_1$, \ldots, $\xi_n$ are all different;}\\
  0&\text{if $\xi_i=\xi_j$ for some $i$, $j$ with $i\neq j$.}
\end{cases}
\]
Then, letting $(L^{\otimes n})^{\Sym_n}$ denote the subalgebra of
$L^{\otimes n}$ fixed under the action of $\Sym_n$ by permutation of
the factors, we set
\[
 C(L/L_0)= \varepsilon^{\otimes n}(e_n^{L_0})\cdot(L^{\otimes
  n})^{\Sym_n}.
\]
We call the algebra $C(L/L_0)$ the \emph{Clifford algebra of
  $L/L_0$}. If $n=1$, then $C(L/L_0) = L$, and the canonical
involution $\underline{\sigma}$ is the canonical involution of $L/L_0$. 

\begin{prop}
\label{prop:prodetale}
If $L/L_0$ and $M/M_0$ are quadratic \'etale extensions,
there is a canonical isomorphism
\[
P\colon C( L \times M/ L_0 \times M_0) \ \iso \ C(
L/L_0) \otimes C(M/M_0),
\]
which corresponds under the Grothendieck correspondence to the
canonical isomorphism of Proposition~\ref{prop:isounion}.
\end{prop}

By contrast to $\Gamma$-sets, we  have an $F$-linear map
\[
c_L\colon L \to C(L/L_0)
\]
defined by
\begin{equation} \label{def:omegaT}
c_L(x) = \varepsilon^{\otimes n}(e_n^{L_0})\cdot\bigl(
(x\otimes1\otimes\cdots\otimes1) + (1\otimes
x\otimes\cdots\otimes1) + (1\otimes\cdots\otimes 1\otimes x)\bigr)
\end{equation}
for $x\in L$.
 
\begin{prop}
\label{productomegaalg.prop}
Let $L/L_0$ and $M/M_0$ be quadratic \'etale extensions and
let $\rho\colon L \times  M \to L \otimes M $ be the linear
map given by $\rho(x,y)= x \otimes 1 + 1\otimes y$. The
diagram
\[
\begin{CD}
  C( L \times M/  L_0\times M_0) @>P>> C(L/L_0) \otimes  C(M/M_0) \\
  @A{ c_{L\times M}}AA @AA{ c_{L} \otimes  c_{M}}A\\
  L \times M @>>\rho> L \otimes M
\end{CD}
\]
is commutative.
\end{prop}

\begin{proof} The claim follows easily from
  the definition of the canonical map $P$.
\end{proof}

\begin{cor} \label{omegagen.cor}
  The space $c_{L}(L)$ generates $C(L/L_0)$ as an algebra.
\end{cor}

\begin{proof}
  By extending $F$ we may assume that ${L/L_0}$ has a nontrivial
  decomposition $L/L_0 =M/M_0  \times N/N_0$. The claim then
  follows from \eqref{productomegaalg.prop} by induction on the
  dimension.
\end{proof}

\begin{lem}
  \label{lem:kerc}
  Let $L_0=\sym(\sigma_L)$ and $T_0\colon L_0\to F$ be the trace
  map. The kernel of $c_L\colon L\to C(L,\sigma_L)$ is
  \[
  \ker c_L=\{x\in L_0\mid T_0(x)=0\}.
  \]
\end{lem}

\begin{proof}
  Extend scalars to $F_s$ and use induction on $\dim L$.
\end{proof}

\subsection{Discriminants}

Let $Q$ be a $2(n-1)$-dimensional $\Gamma$-quadric. The
Clifford $\Gamma$-set $C(Q)$, consisting of maximal linear
subvarieties of $Q$, decomposes into two classes $C_1(Q)$,
$C_2(Q)$: two linear varieties $v_1$, $v_2$ lie in the same class if
and only if $\dim(v_1\cap v_2)\equiv n-1\bmod2$, see for
example~\cite[\S\S II.6]{Dieu}.  
The equivalence relation is clearly preserved under the
action of $\Gamma$, hence $\Gamma$ acts on the set 
\[
\Delta(Q):=\{C^1(Q),C^2(Q)\},
\] 
which we call the \emph{discriminant} of $Q$. We let $\delta\colon
C(Q)\to\Delta(Q)$ denote the $\Gamma$-equivariant map that carries
each linear variety in $C(Q)$ to its class. 

Likewise, if $(Y,\sigma)$ is a thin $\Gamma$-quadric of dimension~$2(n-1)$,
the Clifford $\Gamma$-set $C(Y,\sigma)$ of $(n-1)$-dimensional isotropic
linear subvarieties decomposes into two classes as for ``classical''
quadrics, hence we also have a discriminant $\Delta(Y,\sigma)$ and a
canonical map $\delta\colon C(Y,\sigma) \to \Delta(Y,\sigma)$. An
alternative description of this discriminant is given in
\cite[Proposition~2.5]{KT09}: it is canonically isomorphic to the set
of orbits under the alternating group $\Alt_{2n}$ of the set
\[
\Sigma_{2n}(Y)=\{(y_1,\ldots,y_{2n})\mid Y=\{y_1,\ldots, y_{2n}\}\}.
\]
Therefore, $\Delta(Y,\sigma)$ depends only on $Y$, and not on the
involution $\sigma$; we may thus use the notation $\Delta(Y)$
for $\Delta(Y,\sigma)$.

\begin{prop}
  Let $\varepsilon\colon(Y,\sigma)\hookrightarrow Q$ be an embedding
  of quadrics. There is an induced isomorphism of $\Gamma$-sets
  $\Delta(\varepsilon)\colon \Delta(Y)\to\Delta(Q)$ that fits into the
  following commutative diagram:
  \[
  \begin{CD}
  C(Y,\sigma) @>{C(\varepsilon)}>> C(Q) \\
  @V{\delta}VV @VV{\delta}V \\
  \Delta(Y) @>{\Delta(\varepsilon)}>> \Delta(Q).
\end{CD}
\]
\end{prop}

\begin{proof}
For $\omega_1$, $\omega_2\in C(Y,\sigma)$, we have
\[
\lvert\omega_1\cap\omega_2\rvert = \dim\bigl(C(\varepsilon)(\omega_1)
\cap C(\varepsilon)(\omega_2)\bigr),
\]
hence the map $C(\varepsilon)$ is compatible with the division of
$C(Y,\sigma)$ and $C(Q)$ into two classes. Therefore, we may define
$\Delta(\varepsilon)$ by mapping the equivalence class of $\omega\in
C(Y,\sigma)$ to the equivalence class of $C(\varepsilon)(\omega)\in
C(Q)$.
\end{proof}

Under the Grothendieck correspondence, the discriminant of
$\Gamma$-sets yields a discriminant of \'etale algebras, see
\cite[\S18]{knustignol}. If $(L,\sigma)$ is an \'etale algebra with
involution, we have a canonical map
\[
\delta\colon \Delta(L) \to C(L,\sigma),
\]
so we may consider the Clifford algebra $C(L,\sigma)$ as an \'etale
extension of $\Delta(L)$.

\subsection{Clifford algebra embeddings}
Let $(A,\sigma,f)$ be a central simple algebra of even degree $2n$ with
quadratic pair over an arbitrary field $F$. 
Recall from
\cite[\S8]{KMRT} that the Clifford algebra $C(A,\sigma,f)$ is defined
in such a way that if $A=\End_FV$ and $(\sigma,f)$ is induced by some
quadratic form $q\colon V\to F$, then $C(A,\sigma,f)$ is the even
Clifford algebra  
\[
C(A,\sigma,f)=C_0(V,q).
\]
The center of $C(A,\sigma,f)$ is a quadratic \'etale $F$-algebra
$Z(A,\sigma,f)$, which is related to the discriminant of $(\sigma,f)$,
see \cite[(8.10)]{KMRT}, and $C(A,\sigma,f)$ is an Azumaya algebra of
constant rank $2^n$ over $Z(A,\sigma,f)$. 
The Clifford algebra $C(A,\sigma,f)$ is defined as a 
quotient of the tensor algebra
$T(A)$ of the underlying vector space of $A$, hence there is a
canonical $F$-linear map
\[
c_A\colon A\to C(A,\sigma,f).
\]
In the split case $A=\End_FV$, the bilinear polar form of $q$ yields
an $F$-linear isomorphism $V\to V^*$ which allows us to identify
\[
V\otimes_FV=V\otimes_FV^*=\End_FV;
\]
under this identification, the map $c_A$ carries $v\otimes w\in
V\otimes_FV$ to $v\cdot w\in C_0(V,q)$.

Now, let $\deg A=2n$, and let
\[
\varepsilon\colon(L,\sigma_L)\hookrightarrow(A,\sigma,f)
\]
be an embedding of an \'etale $F$-algebra with involution of the first kind of dimension~$2n$ in $(A,\sigma,f)$.

\begin{thm}
  \label{prop:embedClif}
  There is an $F$-algebra embedding
  \begin{equation} \label{equ:defCepsilon}
  C(\varepsilon)\colon C(L,\sigma_L)\hookrightarrow C(A,\sigma,f)
  \end{equation}
  such that the following diagram commutes:
  \begin{equation}
    \label{diag:embedClif}
    \begin{CD}
      L @>{\varepsilon}>> A\\
      @V{c_L}VV @VV{c_A}V\\
      C(L,\sigma_L) @>{C(\varepsilon)}>> C(A,\sigma,f).
    \end{CD}
  \end{equation}
  The $F$-algebra homomorphism $C(\varepsilon)$ is uniquely determined
  by this property. Moreover, the following diagram commutes:
  \begin{equation}
    \label{diag:invClif}
    \begin{CD}
      C(L,\sigma_L) @>{C(\varepsilon)}>> C(A,\sigma,f) \\
      @V{\underline{\sigma_L}}VV @VV{\underline{\sigma}}V \\
      C(L,\sigma_L) @>{C(\varepsilon)}>> C(A,\sigma,f).
    \end{CD}
  \end{equation}
  Furthermore, there is an $F$-algebra isomorphism
  $\overline{C(\varepsilon)}$ which makes the following diagram
  commute:
  \[
  \begin{CD}
  C(L,\sigma_L) @>{C(\varepsilon)}>> C(A,\sigma,f)\\
  @A{\delta}AA @AAA \\
  \Delta(L) @>{\overline{C(\varepsilon)}}>> Z(A,\sigma,f).
  \end{CD}
  \]
\end{thm}

The proof uses the following general principle:

\begin{principle}
  \label{principle}
  Let $C$ and $C'$ be $F$-algebras, and let $U \subset C$ be
  a subspace which generates $C$ as an $F$-algebra. Let also
  $f\colon\, U \to C'$ be a linear map. If there exists a field
  extension $K/F$ and a $K$-algebra homomorphism $\hat f_K\colon
  C\otimes_FK\to C'\otimes_FK$ such that $\hat f_K\rvert_{U\otimes K}=
  f \otimes \id_{K}$, then $f$ extends to an $F$-algebra homomorphism
  $\hat{f}\colon\, C \to C'$.
\end{principle}

\begin{proof} 
  The restriction of $\hat{f}_{K}$ to $C$ is an $F$-algebra 
  homomorphism $C \to C' \otimes_F K$. Since the restriction of
  $\hat{f}_{K}$ to $U$ is $f$, and $U$ generates $C$, the image
  of this homomorphism is in fact in $C'$.
\end{proof}

\begin{proof}[Proof of Theorem~\ref{prop:embedClif}]
  Since $c(L)$ generates the algebra $C(L,\sigma_L)$, it is clear
  that, if it exists,
  $C(\varepsilon)$ is uniquely determined by the commutativity of
  diagram~\eqref{diag:embedClif}. By \cite[(8.14)]{KMRT}, we have
  \[
  \ker(c_A\circ\varepsilon) = \{x\in L\mid \varepsilon(x)\in
  \sym(\sigma) \text{ and } f\bigl(\varepsilon(x)\bigr)=0\}.
  \]
  Since $\varepsilon$ embeds $(L,\sigma_L)$ in $(A,\sigma,f)$, the
  condition $\varepsilon(x)\in\sym(\sigma_L)$ holds if and only if
  $x\in L_0$, and then
  \[
  f\bigl(\varepsilon(x)\bigr)=\Trd_A\bigl(\varepsilon(\ell x)\bigr)
  =T(\ell x) \qquad\text{for any $\ell\in L$ such that
    $\ell+\sigma_L(\ell)=1$.}
  \]
  Since for $x\in L_0$ we have
  \[
  T(\ell x)=T_0\bigl(\ell x+\sigma_L(\ell x)\bigr) = T_0(x),
  \]
  it follows from Lemma~\ref{lem:kerc} that $\ker
  c_L=\ker(c_A\circ\varepsilon)$. Therefore, $\varepsilon$ induces an
  injective $F$-linear map $\varepsilon'\colon c_L(L)\hookrightarrow
  C(A,\sigma,f)$. Since $c_L(L)$ generates $C(L,\sigma_L)$, we may
  apply the general principle above with $U=c_L(L)$,
  $C=C(L,\sigma_L)$, and $C'=C(A,\sigma,f)$, taking for $K$ a
  separable closure $F_s$ of $F$. To prove the existence of
  $C(\varepsilon)$, it suffices to show that
  $\varepsilon'_s:=\varepsilon\otimes\id_{F_s}$ extends to an
  $F_s$-algebra embedding
  \[
  C(\varepsilon_s)\colon C(L_s,\sigma_{L_s}) \to C(A_s,\sigma_s,f_s).
  \]
  By Proposition~\ref{prop:splitinvembed}, we may assume
  $A_s=\End_{F_s}(L_s)$, $(\sigma_s,f_s)$ is the quadratic pair
  adjoint to the quadratic form $t\colon L_s\to F_s$ of
  Example~\ref{ex:embedinvo}, and $\varepsilon_s$ is the regular
  representation, hence $C(A_s,\sigma_s,f_s)=C_0(L_s,t)$.

  We argue by induction on $n$. If $n=1$ (i.e. $L_s=F_s\times F_s$)
  the map $c_L$ is the identity
  \[
  c_L\colon L_s\iso C(L_s,\sigma_{L_s})=L_s,
  \]
  and $c_A\circ\varepsilon_s\colon L_s\to C_0(L_s,t)$ is an
  $F_s$-algebra isomorphism, so the result is clear.

  If $n>1$ we may find a nontrivial decomposition
  $(L_s,\sigma_s)=(L_1,\sigma_1)\times(L_2,\sigma_2)$. We then have,
  by Proposition~\ref{prop:prodetale},
  \[
  C(L_s,\sigma_s)=C(L_1,\sigma_1)\otimes_FC(L_2,\sigma_2),
  \]
  and, letting $c_i\colon L\to C(L_i,\sigma_i)$ for $i=1$, $2$ denote
  the canonical $F_s$-linear map,
  \[
  c_L(\ell_1,\ell_2)=c_1(\ell_1)\otimes c_2(\ell_2) \qquad\text{for
    $\ell_1\in L_1$ and $\ell_2\in L_2$.}
  \]
  On the other hand, defining quadratic forms $t_1$ and $t_2$ on $L_1$
  and $L_2$ respectively in the same way as $t$ on $L_s$, we have an
  orthogonal sum decomposition
  \[
  (L_s,t)=(L_1,t_1)\stackrel{\perp}{\bigoplus} (L_2,t_2),
  \]
  hence a canonical $F$-algebra embedding
  \[
  C_0(L_1,t_1)\otimes_FC_0(L_2,t_2)\hookrightarrow C_0(L_s,t).
  \]
  By induction, we may assume $(c\circ\varepsilon)_1\colon L_1\to
  C_0(L_1,t_1)$ and $(c\circ\varepsilon)_2\colon L_2\to C_0(L_2,t_2)$
  extend to $F_s$-algebra embeddings
  \[
  C(\varepsilon_1)\colon C(L_1,\sigma_1) \hookrightarrow C_0(L_1,t_1)
  \qquad\text{and}\qquad C(\varepsilon_2)\colon C(L_2,\sigma_2)
  \hookrightarrow C_0(L_2,t_2).
  \]
  Then the composition
  \[
  C(L_s,\sigma_{L_s}) = C(L_1,\sigma_1) \otimes_F C(L_2,\sigma_2)
  \xrightarrow{C(\varepsilon_1)\otimes C(\varepsilon_2)} 
  C_0(L_1,t_1)\otimes_FC_0(L_2,t_2) \hookrightarrow C_0(L_s,t)
  \]
  is an $F_s$-algebra embedding that extends $c_A\circ\varepsilon_s$.

  To prove that diagram~\eqref{diag:invClif} commutes, it suffices to
  prove the commutativity of
  \[
  \begin{CD}
    c_L(L) @>{C(\varepsilon)}>> c_A(A)\\
    @V{\underline{\sigma_L}}VV @VV{\underline{\sigma}}V\\
    c_L(L) @>{C(\varepsilon)}>> c_A(A)
  \end{CD}
  \]
  since $c_L(L)$ generates $C(L,\sigma_L)$. Commutativity of the
  latter diagram follows from commutativity of the diagrams
  \[
  \begin{CD}
    L @>{\varepsilon}>> A\\
    @V{\sigma_L}VV @VV{\sigma}V\\
    L @>{\varepsilon}>> A,
  \end{CD}
  \qquad
  \begin{CD}
    L @>{c_L}>> C(L,\sigma_L) \\
    @V{\sigma_L}VV @VV{\underline{\sigma_L}}V\\
    L @>{c_L}>> C(L,\sigma_L),
  \end{CD} 
  \quad\text{and}\quad
  \begin{CD}
    A @>{c_A}>> C(A,\sigma,f)\\
    @V{\sigma}VV @VV{\underline{\sigma}}V\\
    A @>{c_A}>> C(A,\sigma,f).
  \end{CD}
  \]
  
  To complete the proof, it remains to show that $C(\varepsilon)$ maps
  $\Delta(L)$ to $Z(A,\sigma,f)$. It suffices to see that the image of
  $\Delta(L)$ centralizes $c_A(A)$ since $c_A(A)$ generates
  $C(A,\sigma,f)$. Again, we may extend scalars to reduce to the case
  where $L$ is split, and use induction on $\dim L$. Details are left
  to the reader.
\end{proof}

\begin{remark} The situation described in Theorem~\ref{prop:embedClif} 
  is related to results about Cartan subspaces given in \cite{Tits68}.
  Assume that $F$ has characteristic different from $2$.
  A \emph{Cartan subspace} of a central simple algebra $A$ of even
  degree $2n$ with an orthogonal involution $\sigma$ is by definition
  a subspace of $\Skew(A,\sigma_A)$ such that there exists an element $a
  \in H$ which is invertible in $A$, is separable (i.e., generates an \'etale $F$-subalgebra
  of $A$) and such that $a^{-1}H$ is an \'etale commutative subalgebra
  of $A$ of dimension $2n$.  Equivalently, if $a \in \Skew(A,\sigma_A)$ is
  invertible and separable in $A$, then $H = \{x \in \Skew(A,\sigma_A) \ 
  \vert \ xa = ax\}$ is a Cartan subspace for~$a$. It can be shown
  that the algebra $L_0=a^{-1}H$ is independent of the choice of the
  element $a$.  Moreover the subalgebra $L =L_0 \oplus H$ of $A$ is
  maximal commutative \'etale in $A$ and has an involution $\sigma_L$
  which is the restriction of $\sigma_A$ to $L$.  Thus the set of
  skew-symmetric elements in $L$ is a Cartan subspace of A. It follows
  from \cite[Th\'eor\`eme 
  3]{Tits68} that $c_A(L)$ generates a maximal commutative \'etale
  subalgebra of $C(A,\sigma_A)$. This algebra is isomorphic to
  $C(L,\sigma_L)$. 
\end{remark}

\subsection{The Severi--Brauer variety of a Clifford algebra}

We discuss in this subsection the geometric analogue of
Theorem~\ref{prop:embedClif}, showing that an embedding of a thin
quadric into a quadric induces an embedding of the corresponding
Clifford $\Gamma$-sets.

Recall that the Severi--Brauer variety of a central simple algebra is a
$\Gamma$-projective space. To take in account the fact that
$C(A,\sigma,f)$ is an Azumaya algebra of constant rank over the
\'etale $F$-algebra $Z(A,\sigma,f)$ we  
introduce a variant of the Severi--Brauer variety for Azumaya algebras
over quadratic \'etale algebras.\footnote{For the sake of simplicity we
  give an ad-hoc construction.}  We consider disjoint
unions $P_1\sqcup P_2$ of two 
projective $F$-spaces of the same dimension, endowed with a continuous
action of a profinite group $\Gamma$ by collineations that may
interchange the two components. We call these unions \emph{double
  $\Gamma$-projective spaces}. They are equipped with a canonical
structure map to a $\Gamma$-set $Z$ of two elements
\[
Z:=\bigl\{\{P_1\},\{P_2\}\bigr\} \leftarrow P_1\sqcup P_2.
\]
If $\dim P_1=\dim P_2=n-1$, we may define for every integer
$k\in\n$ two different sets:
\begin{multline*}
gr_{k-1}(P_1\sqcup P_2)= Gr_{k-1}(P_1)\cup Gr_{k-1}(P_2) \qquad\text{and}\\
Gr_{k-1}(P_1\sqcup P_2) = \{v_1\sqcup v_2 \mid v_1\in Gr_{k-1}(P_1),\;
v_2\in Gr_{k-1}(P_2)\}.
\end{multline*}
Note that, even though $Gr_{k-1}(P_1)$ and $Gr_{k-1}(P_2)$ are not
$\Gamma$-sets since they may be interchanged under the
$\Gamma$-action, $\Gamma$ naturally acts on $gr_{k-1}(P_1\sqcup P_2)$
and $Gr_{k-1}(P_1\sqcup P_2)$, and $gr_{k-1}(P_1\sqcup P_2)$ carries a
canonical $\Gamma$-equivariant map
\[
Z\leftarrow gr_{k-1}(P_1\sqcup P_2).
\]
If $\Gamma_0\subset\Gamma$ is the subgroup fixing $Z$, we clearly have
an isomorphism of $\Gamma_0$-sets
\[
Gr_{k-1}(P_1\sqcup P_2) = Gr_{k-1}(P_1)\times Gr_{k-1}(P_2).
\]
If $\Gamma_0\neq\Gamma$, then $Gr_{k-1}(P_1\sqcup P_2)$ may be identified
with the norm of the $\Gamma_0$-set $Gr_{k-1}(P_1)$ (or $Gr_{k-1}(P_2)$)
following Ferrand's general definition in \cite{ferrand}.

For $(Y\xleftarrow{\pi} X)$ a
$\Gamma$-covering with $\lvert Y\rvert=2$ and $\lvert X\rvert=2n$, a
\emph{general embedding}
\[
\varepsilon\colon (Y\xleftarrow{\pi} X) \hookrightarrow (Z\leftarrow
P_1\sqcup P_2)
\]
is a pair of $\Gamma$-equivariant maps $\varepsilon\colon X\to
gr_0(P_1\sqcup P_2)$, $\overline{\varepsilon}\colon Y\to Z$ such that
\begin{enumerate}
\item[(a)]
the following diagram commutes:
\[
\begin{CD}
  X @>{\varepsilon}>> gr_0(P_1\sqcup P_2)\\
  @V{\pi}VV @VVV\\
  Y @>{\overline{\varepsilon}}>> Z,
\end{CD}
\]
\item[(b)]
$\overline{\varepsilon}$ is bijective, and
\item[(c)]
the image of each fiber of $Y$ is a set of $n$ points in general
position in $P_1$ and $P_2$ respectively.
\end{enumerate}
The embedding $\varepsilon$ thus induces $\Gamma$-equivariant maps
\[
Gr_{k-1}\varepsilon\colon Gr_{k-1}(X/Y) \hookrightarrow
gr_{k-1}(P_1\sqcup P_2) \qquad\text{for all $k\in\n$.}
\]
These maps fit into commutative diagrams
\[
\begin{CD}
  Gr_{k-1}(X/Y) @>{\lambda^k(\varepsilon)}>> gr_{k-1}(P_1\sqcup P_2)\\
  @VVV @VVV\\
  Y @>{\overline{\varepsilon}}>> Z.
\end{CD}
\]
\medbreak
\par
Double $\Gamma$-projective space arise from the
following construction. Let $Z$ be a quadratic \'etale algebra over an
arbitrary field $F$, 
and let $B$ be an Azumaya $Z$-algebra of constant rank. As usual, let
$F_s$ denote a separable closure of $F$, let $Z_s=Z\otimes_FF_s$ and
$B_s=B\otimes_FF_s$ the $F_s$-algebras obtained by scalar extension,
and let $\Gamma$ be the Galois group of $F_s/F$. If $\dim_FB=2n^2$,
then $B_s\simeq M_n(F_s)\times M_n(F_s)$; its minimal right ideals
have dimension $n$. We let $\DSB(B)$ be the double $\Gamma$-projective
space whose points are the $n$-dimensional right ideals of $B_s$, with
the natural $\Gamma$-action. These ideals are divided into two
classes, which correspond to the two primitive idempotents of $Z_s$:
to each $n$-dimensional right ideal $I$, we associate the primitive
idempotent $e_I\in Z_s$ such that $e_I\cdot I=I$. Alternatively, we
associate to $I$ the element $I^\delta\in \funcX(Z)$ such that
$I^\delta(e_I)=1$. Thus, the canonical structure map of the double
$\Gamma$-projective space $\DSB(B)$ can be identified with
\[
\funcX(Z)\xleftarrow{\delta} \DSB(B).
\]

Now, let $K\xrightarrow{i} L$ be an \'etale extension of degree~$n$ of
a quadratic \'etale $F$-algebra. (Thus, $\dim_FL=2n$.) By definition,
an embedding
\[
\varepsilon\colon (L/K) \hookrightarrow (B/Z)
\]
is an $F$-algebra embedding $L\hookrightarrow B$ that restricts to an
isomorphism $\overline{\varepsilon}\colon K\iso Z$.

\begin{prop}
  \label{prop:doubleregrep}
  If, in the situation above, $L$ is a split \'etale $F$-algebra, then
  for every embedding $\varepsilon\colon (L/K)\hookrightarrow (B/Z)$
  there 
  is an $F$-algebra isomorphism $\varphi\colon B\iso\End_KL$ whose
  restriction to $Z$ is $\overline{\varepsilon}^{-1}$ such that
  $\varphi\circ\varepsilon\colon L\hookrightarrow \End_KL$ is the
  regular representation.
\end{prop}

\begin{proof}
  By hypothesis $K$ is split, hence $Z$ also. Using
  $\overline{\varepsilon}$ as an identification, let
  \[
  K=Z=F\times F,
  \]
  and consider the corresponding decompositions $L=L_1\times L_2$,
  $B=B_1\times B_2$, and
  $\varepsilon=(\varepsilon_1,\varepsilon_2)$. By
  Proposition~\ref{prop:splitembed}, we may find $F$-algebra
  isomorphisms $\varphi_i\colon B_i\iso\End_FL_i$, for $i=1$, $2$,
  such that $\varphi_i\circ\varepsilon_i\colon L_i\hookrightarrow
  \End_FL_i$ is the regular representation. The proposition holds with
  $\varphi=(\varphi_1,\varphi_2)\colon B\to \End_FL_1\times \End_FL_2
  = \End_KL$.
\end{proof}

Now, consider an arbitrary embedding $\varepsilon\colon (L/K)
\hookrightarrow (B/Z)$ of an \'etale extension of degree~$n$. For
every $\xi\in\funcX(L)$ we let as in\eqref{eq:epsstar}
\[
\varepsilon_*(\xi)=\{x\in B_s\mid (1\otimes\xi(\ell))\cdot x =
(\varepsilon(\ell)\otimes1)\cdot x \text{ for all $\ell\in L$}\}.
\]

\begin{prop}
  \label{prop:embedouble}
  The right ideal $\varepsilon_*(\xi)$ is $n$-dimensional, and
  $\varepsilon_*$ defines a general embedding
  \[  \varepsilon_*\colon \bigl(\funcX(K)\xleftarrow{\funcX(i)} \funcX(L)\bigr)
  \hookrightarrow \bigl(\funcX(Z)\leftarrow\DSB(B)\bigr).
  \]
  Conversely, every general embedding of a $\Gamma$-covering of
  degree~$n$ with $2$-element base in $\DSB(B)$ arises from an
  embedding of an \'etale extension in $B$.
\end{prop}

\begin{proof}
  The proof of the first (resp.\ second) part is similar to that of
  Proposition~\ref{prop:embed} (resp.\ \ref{prop:embedconv}).
\end{proof}

Proposition~\ref{prop:embedouble} applies in particular in the
situation of Theorem~\ref{prop:embedClif}: for the rest of this
subsection, suppose $A$ is a central simple $F$-algebra of degree~$2n$
with a quadratic pair $(\sigma,f)$. If $(L,\sigma_L)$ is an \'etale
$F$-algebra with involution of dimension~$2n$ and
\[
\varepsilon\colon (L,\sigma_L) \hookrightarrow (A,\sigma,f)
\]
is an embedding of algebras with involution, then
Theorem~\ref{prop:embedClif} yields an $F$-algebra embedding of the
\'etale extension $C(L,\sigma_L)/\Delta(L)$ into
$C(A,\sigma,f)/Z(A,\sigma,f)$. By Proposition~\ref{prop:embedouble}
there is a general embedding
\begin{equation}
\label{eq:Ceps}
C(\varepsilon)_*\colon \funcX C(L,\sigma_L) \hookrightarrow
\DSB\bigl(C(A,\sigma,f)\bigr).
\end{equation}
On the other hand, the embedding $\varepsilon$ also yields by
Proposition~\ref{prop:embedquad} an embedding of quadrics
\begin{equation}
  \label{eq:epstar}
  \varepsilon_*\colon\funcY(L,\sigma_L) \hookrightarrow Q(\sigma,f).
\end{equation}
To complete the picture, we now relate the Clifford $\Gamma$-set
$C\bigl(Q(\sigma,f)\bigr)$ to the Severi--Brauer double projective
space $\DSB\bigl(C(A,\sigma,f)\bigr)$.

Recall from \cite[\S8]{KMRT} that there is a
Clifford bimodule $B(A,\sigma,f)$ over $C(A,\sigma,f)$, which is
equipped with a canonical $F$-linear map $b\colon A\to B(A,\sigma,f)$,
see \cite[\S9]{KMRT}.
In the split case $A=\End_FV$, we have
$B(A,\sigma,f)=V\otimes_FC_1(V,q)$. The bilinear polar form of $q$
yields an $F$-linear isomorphism $V\to V^*$ which allows us to
identify 
\[
V\otimes_FV=V\otimes_FV^*=\End_FV;
\]
under this identification, the map $b$ carries $v\otimes w\in
V\otimes_FV$ to $v\otimes w\in V\otimes_FC_1(V,q)$.

Since $\deg A=2n$, we have
$\dim_FC(A,\sigma,f)=2^{2n-1}$. For every $2n$-dimensional isotropic
right ideal $I\subset A_s$, let $\psi(I)\subset C(A_s,\sigma_s,f_s)$
denote the right annihilator of the $F$-vector space
$\bigl(\sigma_s(I)\bigr)^b\subset B(A_s,\sigma_s,f_s)$, i.e.,
\[
\psi(I)=\{x\in C(A_s,\sigma_s,f_s)\mid \bigl(\sigma_s(I)\bigr)^b\cdot
x=0 \text{ for all $y\in I$}\}.
\]

\begin{prop}
  \label{prop:doubleClif}
  The map $\psi$ defines an embedding
  \[
  \psi\colon \bigl(\Delta(Q(\sigma,f))\leftarrow C(Q(\sigma,f))\bigr)
  \hookrightarrow \bigl(\funcX(Z(A,\sigma,f)) \leftarrow
  \DSB(C(A,\sigma,f))\bigr). 
  \]
\end{prop}

\begin{proof}
  We first show $\dim_F\psi(I)=2^{n-1}$. Let $A_s=\End_{F_s}V$ for
  some $2n$-dimensional $F_s$-vector space $V$. The quadratic pair
  $(\sigma_s,f_s)$ is adjoint to some hyperbolic quadratic form $q$ on
  $V$, and $I=\Hom_{F_s}(V,U)$ for some totally isotropic
  $n$-dimensional subspace $U\subset V$. We may therefore identify $V$
  with $U\oplus U^*$, and $q$ with the quadratic form
  $q(u,\varphi)=\varphi(u)$ for $u\in U$, $\varphi\in U^*$. Then the
  full Clifford algebra of $q$ can be identified with the algebra
  $\End_{F_s}(\bigwedge U)$ of $F_s$-linear endomorphisms of the
  exterior algebra of $U$, in such a way that the canonical map
  \[
  V=U\oplus U^*\to C(V,q)=\End_{F_s}({\textstyle\bigwedge} U)
  \]
  maps $u\in U$ to the left exterior multiplication by $u$ and
  $\varphi\in U^*$ to its associated derivation $d_\varphi$,
  see~\cite[(8.3)]{KMRT}. Let
  \[
  \textstyle{\bigwedge_0}U = \bigoplus_{\text{$i$ even}}
  {\textstyle\bigwedge^i} U\subset 
  {\textstyle\bigwedge} U \qquad\text{and}\qquad
  {\textstyle\bigwedge_1}U = 
  \bigoplus_{\text{$i$ odd}} {\textstyle\bigwedge^i} U\subset U.
  \]
  Then
  \[
  C(A_s,\sigma_s,f_s) = C_0(V,q) = \End_{F_s}({\textstyle\bigwedge_0}
  U) \times 
  \End_{F_s}({\textstyle\bigwedge_1}U)
  \]
  and
  \[
  B(A_s,\sigma_s,f_s) = V\otimes_{F_s} C_1(V,q) = V\otimes_{F_s}
  \bigl(\Hom_{F_s}({\textstyle\bigwedge_0U,\bigwedge_1}U) \times
  \Hom_{F_s}({\textstyle\bigwedge_1U,\bigwedge_0}U)\bigr).
  \]
  Under the identification $A_s=\End_{F_s}V=V\otimes_{F_s}V$ we have
  $I=U\otimes_{F_s}V$, hence
  \[
  \bigl(\sigma_s(I)\bigr)^b = V\otimes_{F_s}U \subset V\otimes_{F_s}
  C_1(V,q).
  \]
  Therefore, $\psi(I)$ consists of the linear maps
  $f\in\End_{F_s}(\bigwedge U)$ such that $f(\bigwedge_0U)\subset
  \bigwedge_0U$, $f(\bigwedge_1U)\subset \bigwedge_1U$, and
  \[
  u\wedge f(x)=0 \qquad\text{for all $u\in U$ and all
    $x\in{\textstyle\bigwedge} 
    U$.}
  \]
  The last condition implies $f(x)\in \bigwedge^nU$, hence
  \[
  \psi(I)=
  \begin{cases}
    \Hom_{F_s}(\bigwedge_0U,\bigwedge^nU) &\text{if $n$ is even,} \\
    \Hom_{F_s}(\bigwedge_1U,\bigwedge^nU) &\text{if $n$ is odd.}
  \end{cases}
  \]
  It follows that $\dim_{F_s}\psi(I)=\dim_{F_s}\bigwedge_0U =
  \dim_{F_s}\bigwedge_1U=2^{n-1}$, hence
  $\psi(I)\in\DSB\bigl(C(A,\sigma,f)\bigr)$. Moreover, for every
  nonzero $\varphi\in U^*$ we have $d_\varphi(\bigwedge^nU)\neq\{0\}$,
  hence
  \[
  \{v\in V\mid v\cdot f=0\text{ in $C(V,q)$ for all $f\in\psi(I)$}\} =
  U. 
  \]
  It follows that
  \[
  I=\{y\in A\mid \bigl(\sigma(y)\bigr)^b\cdot x = 0 \text{ for all
    $x\in \psi(I)$}\},
  \]
  hence the map $\psi$ is injective.

  It is clear from the definition that the map $\psi$ is
  $\Gamma$-equivariant. To complete the proof, it remains to show that
  $\psi$ induces an isomorphism
  \[
  \overline{\psi}\colon\Delta\bigl(Q(\sigma,f)\bigr) \iso
  X\bigl(Z(A,\sigma,f)\bigr).
  \]
  As above, we identify $A_s=\End_{F_s}V$ for some $2n$-dimensional
  $F_s$-vector space $V$. Let $I$, $I'\in C\bigl(Q(\sigma,f)\bigr)$,
  so
  \[
  I=\Hom_{F_s}(V,U), \qquad I'=\Hom_{F_s}(V,U')
  \]
  for some totally isotropic $n$-dimensional subspaces $U$, $U'\subset
  V$. By a theorem of Witt, there is an orthogonal transformation
  $g\in\Orth(V,q)$ such that $g(U)=U'$. This transformation induces an
  automorphism $\hat g$ of $C_0(V,q)$ such that $\hat
  g\bigl(\psi(I)\bigr)=\psi(I')$. If $e$ (resp.\ $e'$) is the
  primitive central idempotent of $C_0(V,q)$ such that
  $e\cdot\psi(I)=\psi(I)$ (resp.\ $e'\cdot \psi(I')=\psi(I')$), we
  therefore have $\hat g(e)=e'$. Now, we have $\dim(U\cap U')\equiv
  n\bmod2$ if and only if $g\in\Orth^+(V,q)$ by \cite[\S\S~II.6,
  II.10]{Dieu}, and this condition holds if and only if the
  restriction of $\hat g$ to the center of $C_0(V,q)$ is the identity,
  by \cite[(13.2)]{KMRT}. Therefore, $e=e'$ if and only if the ideals
  $I$ and $I'$ belong to the same class of
  $C\bigl(Q(\sigma,f)\bigr)$. This completes the proof.
\end{proof}

\begin{cor}
  If $(\sigma,f)$ is a hyperbolic quadratic pair on $A$, then the
  Clifford algebra $C(A,\sigma,f)$ decomposes into a direct product of
  simple $F$-algebras $C(A,\sigma,f)\simeq C_1\times C_2$, and one of
  the factors $C_1$, $C_2$ is split.
\end{cor}

\begin{proof}
  If $(\sigma,f)$ is hyperbolic, then its discriminant is trivial,
  hence $Z(A,\sigma,f)\simeq F\times F$ and $C(A,\sigma,f)$ decomposes
  into a direct product of simple $F$-algebras as above. Moreover,
  there is a maximal isotropic linear variety in $Q(\sigma,f)$ that is
  defined over $F$, hence fixed under the action of $\Gamma$. Its
  image under $\psi$ is a point on
  $\DSB\bigl(C(A,\sigma,f)\bigr)=\SB(C_1)\sqcup \SB(C_2)$ fixed under
  the action of $\Gamma$. Its existence shows that $C_1$ or $C_2$ is
  split. 
\end{proof}

\begin{rem}
  This corollary was first proved by D.C. Van Drooge and H.P. Allen,
  see also Tits \cite[Proposition~8]{Tits68}.
\end{rem}

We may now relate the maps \eqref{eq:Ceps} and
\eqref{eq:epstar} associated to an embedding $\varepsilon$ of algebras
with involution:

\begin{prop}
  The following diagram is commutative:
  \begin{equation} \label{eq:diagrammclifford}
  \begin{CD}
    C\funcY(L,\sigma_L) @>{C(\varepsilon_*)}>> C\bigl(Q(\sigma,f)\bigr)\\
    @| @VV{\psi}V\\
    \funcX C(L,\sigma_L) @>{C(\varepsilon)_*}>>
    \DSB\bigl(C(A,\sigma,f)\bigr). 
  \end{CD}
  \end{equation}
\end{prop}

\begin{proof}
  Extending scalars, we reduce to the case where $L$ is split,
  $A=\End_FL$, and $\varepsilon$ is the regular representation as in
  Example~\ref{ex:embedinvo} (see
  Proposition~\ref{prop:splitinvembed}). Let $\omega\in
  C\funcY(L,\sigma_L)$, and let $e_1$, \ldots, $e_n$, $e'_1$, \ldots,
  $e'_n$ be the primitive idempotents of $L$, labelled in such a way
  that $\sigma(e_i)=e'_i$ for $i\in\n$ and $\xi(e_i)=1$, $\xi(e'_i)=0$
  for $\xi\in\omega$. Under the identification $C\funcY(L,\sigma_L) =
  XC(L,\sigma_L)$, the corresponding $F$-algebra homomorphism
  $\xi_\omega\colon C(L,\sigma_L)\to F$ satisfies
  \[
  \xi_\omega\bigl(c_L(e_i)\bigr) = 1 \quad\text{and}\quad
  \xi_\omega\bigl(c_L(e'_i)\bigr)=0\quad\text{for all $i\in\n$.}
  \]
  In $C(A,\sigma,f)=C_0(L,t)$ we have $\varepsilon(e_i)=e_ie'_i$ and
  $\varepsilon(e'_i)=e'_ie_i$ for $i\in\n$, hence
  \[
  C(\varepsilon)_*(\xi_\omega)= \{x\in
  C_0(L,t)\mid e_ie'_ix=x\text{ and }e'_ie_ix=0\text{ for all
  $i\in\n$}\}. 
  \]
  On the other hand, $C(\varepsilon_*)(\omega) =
  \Hom_F(L,\bigoplus_{i\in\n} e_iL)$, hence
  \[
  \psi\bigl(C(\varepsilon_*)(\omega)\bigr) = \{x\in C_0(L,t) \mid
  e_i\cdot x=0 \text{ for $i\in\n$}\}.
  \]
  In $C(L,t)$, we have $e_i^2=0$ and $e_ie'_i+e'_ie_i=1$,
  hence $e_ie'_ie_i=e_i$ for $i\in\n$. Therefore, if $e'_ie_ix=0$,
  then $e_ix=0$. Conversely, if $e_ix=0$, then $e'_ie_ix=0$ and
  \[
  x=(e_ie'_i+e'_ie_i)x=e_ie'_ix.
  \]
  Therefore, $\psi\bigl(C(\varepsilon_*)(\omega)\bigr) =
  C(\varepsilon)_*(\omega)$. 
\end{proof}

\section{Dynkin diagrams and Severi--Brauer varieties}
\label{sec:Dynkin}

The projective linear group $\PGL_n$ has Dynkin diagram of type
$A_{n-1}$ and the group $\Sym_n$ 
is its Weyl group. The special orthogonal group  $\PGO^+_{2n}$, with Weyl group
$ \Sym_2^{n-1} \rtimes \Sym_n$, has Dynkin diagram of type $D_n$.
Accordingly we call groupoids occurring in Proposition~\ref{prop:SBequiv}
and \eqref{eq:anteq} \emph{groupoids of type $A_{n-1}$},
and groupoids occurring in \eqref{eq:QEtexQCov} and
Proposition~\ref{prop:QCSAQuad}\emph{ groupoids of type $D_{n}$}. Many
properties of these groupoids are related to properties of 
the corresponding Dynkin diagrams. After introducing oriented quadrics 
and \'etale algebras,
we briefly consider two examples: the equivalence $A_3 =D_3$ and triality.

\subsection{Special quadrics}

A \emph{special or oriented $\Gamma$-quadric} is a pair $(Q,\partial)$,
where $Q$ is $\Gamma$-quadric (or a thin $\Gamma$-quadric) and $\partial$
is a fixed isomorphism of
$\Gamma$-sets $\boldsymbol{2} \liso \Delta(Q)$.
In particular the
$\Gamma$-action on $\Delta(Q)$ is trivial. There are two possible
choices for $\partial$. A choice is an \emph{orientation of\,~$Q$}.
If $Q$ is oriented we may use the given isomorphism 
$\partial\colon  \boldsymbol 2 \liso \Delta(Q)$ to decompose the Clifford
$\Gamma$-set:
\[
 C(Z/Z_0 ) = C_1(Z/Z_0,\partial) \sqcup C_2(Z/Z_0,\partial) 
\]
where
\[
  C_i(Z/Z_0,\partial) =\{\omega\in C(Z/Z_0)\mid
  \omega^{\delta\partial}=i\}\qquad\text{for $i=1$, $2$}.
\]
  
Similarly, we define \emph{oriented extensions of \'etale
  $F$-algebras} as pairs 
$(L/L_0,\partial)$ where $L/L_0$ is an extension of \'etale $F$-algebras
and $\partial$ is an orientation of $L$, i.e., a fixed isomorphism
$\partial \colon \Delta(L) \iso F\times F$.
We also have 
\emph{oriented central simple algebras with quadratic pairs}, defined
as $4$-tuples $(A,\sigma,f, \partial)$, where $\partial$ is a fixed
isomorphism $\partial\colon Z(A,\sigma, f) \iso F\times F$.
In the corresponding decomposition
\[
 C(A,\sigma,f, \partial) = 
 C_1(A,\sigma,f, \partial)  \times C_2(A,\sigma,f, \partial), 
\]
$C_1(A,\sigma,f, \partial) $ and $C_2(A,\sigma,f, \partial)$ 
are central simple algebras over $F$. The classes of morphisms
considered in the previous sections can be restricted to preserve the
orientations. For example an embedding of oriented quadrics 
\[
\varepsilon \colon (Y, \sigma, \partial_Y) \hookrightarrow (Q,\partial_Q)
\]
is an embedding $\varepsilon \colon Y \hookrightarrow Q$ such that  
$\Delta(\varepsilon)\circ\partial_Q = \partial_Y$.  The
diagram~\eqref{eq:diagrammclifford} restricts to a commutative diagram  
  \[
  \begin{CD}
    C_1\funcY (L/L_0,\partial_L) @>{C_1(\varepsilon_*)}>>
    C_1\bigl(Q(\sigma,f),\partial_Q\bigr)\\ 
    @| @VV{\psi_1}V\\
    \funcX C_1(L/L_0,\partial_L) @>{C_1(\varepsilon)_*}>>
    \SB\bigl(C_1(A,\sigma,f,\partial)\bigr). 
  \end{CD}
 \]
As in~\ref{prop:SBequiv}, \ref{eq:anteq},
\ref{eq:QEtexQCov} and~\ref{prop:QCSAQuad}, we have
anti-equivalences of groupoids 
\begin{equation*} 
\OQCSA_F^{2n} \equiv \OQuad_\Gamma^{2(n-1)}
\quad \text{and} \quad \OQEtex^n_F \equiv \OQCov^n_\Gamma,
\end{equation*}
using the following notations:
\begin{enumerate}
\item $\OQCSA_F^{2n}$ for the groupoid of oriented central simple algebras over $F$ of degree $2n$  with quadratic pairs, 
\item $\OQuad_\Gamma^{2(n-1)}$ for the groupoid of oriented
  $\Gamma$-quadrics of dimension $2(n-1)$ over~$F$,
\item  $\OQEtex^n_F $ for the groupoid of oriented quadratic \'etale
  extensions $L/L_0$ of $F$-algebras with $\dim_FL_0 =n$. 
\item $\OQCov^n _\Gamma$ for the groupoid of oriented thin $\Gamma$-quadrics of
dimension $2(n-1)$ over~$F$, i.e., double $\Gamma$-coverings $Y/Y_0$
with $\lvert Y_0\rvert=n$.
\end{enumerate} 

Isomorphism classes are in bijection with $H^1(\Gamma, \PGO^+_{2n})$,
resp.\ $H^1(\Gamma, \Sym_2^{n-1}\rtimes \Sym_n)$.

\subsection{The exceptional isomorphism $A_3 =D_3$}

The fact that the Dynkin diagrams $A_3$ and  $D_3$ are identical 
can be interpreted at the level of the corresponding groupoids. For
any $4$-element $\Gamma$-set $X$, the $\Gamma$-set $Gr_1(X)$ of
$2$-element subsets of $X$ is equipped with a canonical involution
$\sigma$, which maps each $2$-element subset of $X$ to its
complement. Thus, it defines a functor
$\Set_\Gamma^4\to\QCov_\Gamma^3$, which we also designate by
$Gr_1$. The sections $\{S_1,S_2,S_3\}\in C\big(Gr_1(X),\sigma\big)$ (where
$S_1$, $S_2$, $S_3$ are $2$-element subsets of $X$ such that
$\{\sigma(S_1),\sigma(S_2),\sigma(S_3)\}\cap\{S_1,S_2,S_3\}=\emptyset$)
naturally fall into two classes, depending on whether $S_1\cap S_2\cap
S_3$ is empty or is a $1$-element subset of $X$. Therefore, there is a
natural orientation on $Gr_1(X)$, and we may view $Gr_1$ as a functor
$\Set_\Gamma^4\to\OQCov_\Gamma^3$. 

The analogue of $Gr_1$ for \'etale algebras under the Grothendieck
correspondence is an ``exterior algebra'' functor $\lambda^2\colon
\Et_F^4\to \OQEtex^3_F$ studied in \cite{knustignol}. There is also an
``exterior algebra'' functor $\lambda^2\colon\CSA_F^4\to \OQCSA_F^6$
discussed in \cite[\S15D]{KMRT}.

\begin{prop}
The functors $Gr_1$ (resp.\ $\lambda^2$) and $C_1$ induce equivalences:
\begin{align*}
\xymatrix@1@M=8pt{\Set_\Gamma^4 \ar@<0.5ex>[r]^{Gr_1}& \OQCov^3_\Gamma,
  \ar@<0.5ex>[l]^{C_1}} &\qquad
\xymatrix@1@M=8pt{\Et_F^4 \ar@<0.5ex>[r]^{\lambda^2}& \OQEtex^3_F,
  \ar@<0.5ex>[l]^{C_1}}\\
\xymatrix@1@M=8pt{\Proj_\Gamma^3 \ar@<0.5ex>[r]^{Gr_1}& \OQuad^4_\Gamma,
  \ar@<0.5ex>[l]^{C_1}} &\qquad
\xymatrix@1@M=8pt{\CSA_F^4 \ar@<0.5ex>[r]^{\lambda^2}& \OQCSA^6_\Gamma.
  \ar@<0.5ex>[l]^{C_1}}
\end{align*}
Moreover these equivalences are compatible with general embeddings.
\end{prop}  

\begin{proof}
The (equivalent) first two cases are discussed in \cite{knustignol},
the third is equivalent to the last and the last follows from
\cite[(15.32)]{KMRT}. 
\end{proof}

\begin{remark}
The equivalence between groupoids of thin objects of type $A_3$ and  of type
$D_3$  is induced at the cohomological level  by the well-known
isomorphism
$\Sym_4 \isom \Sym_2^2 \rtimes \Sym_3 $. Classically, for any
separable field extension $E$ of degree $4$, the subalgebra of
$\lambda^2E$ fixed under the involution is the \emph{cubic resolvent}
(see \cite{knustignol}).  On
the geometric side, the quadric $Gr_1(P)$ can be viewed as a
$4$-dimensional quadric
parametrised by a $3$-dimensional projective space. This quadric is
known as \emph{Klein's quadric}.
\end{remark}

\subsection{Triality}
All Dynkin diagrams but one admit at most
automorphisms of 
order two, which are related to duality in algebra and geometry.
The Dynkin diagram 
of $D_4$:\\[1ex]
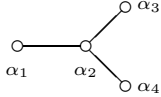
\[
\unitlength =0.5ex
\begin{picture}(-10,0)%
        \put(11,7){$\vcenter{\hbox{$\scriptstyle \alpha _3$}}$}%
        \put(11,-7){$\vcenter{\hbox{$\scriptstyle \alpha_ 4
\phantom{\scriptstyle -1}$}}$}%
\end{picture}%
\vcenter{\hbox{\begin{picture}(0,0)%
      \put(0,0){\circle{2}}%
      \put(0,-5){\hcenter{$\scriptstyle \alpha _1$}}
        \put(1,0){\line(1,0){10}}%
        \put(12,0){\circle{2}}%
        \put(12,-5){\hcenter{$\scriptstyle \alpha_ 2$}}%
                \put(19,7){\circle{2}}%
        \put(19,-7){\circle{2}}%
      \put(12.71,0.71){\line(1,1){5.58}}
        \put(12.71,-0.71){\line(1,-1){5.58}}
              \end{picture}}}
\]\\[0.1ex]
is special, in the sense that it admits automorphisms of order $3$.
Algebraic and geometric objects related to $D_4$ are of particular
interest as they also usually admit exceptional automorphisms of order
$3$, which are called trialitarian. This is the case for the four groupoids
$\OQCov^4_\Gamma$, $\OQEtex_F^4$, $\OQuad^6_\Gamma$ and 
$\OQCSA^8_F$ of type $D_4$. 
For $(E, \partial_E)$ in  $\OQEtex_F^4$ or in 
$\OQCSA_F^8$,  the objects $C_1(E,\partial_E)$ and
$C_2(E,\partial_E)$ belong to  $\QEtex_F^4$, resp.\ to
$\QCSA^8_F$ (see~\cite{KT09}, resp. \cite[\S35]{KMRT}).
 Corresponding results  hold for 
objects in $\OQCov^4_\Gamma$ or $\OQuad^6_\Gamma$ since these
groupoids of quadrics are anti-equivalent to the corresponding
groupoids of algebras. 
Isomorphism classes of these groupoids are in bijection with
$H^1(\Gamma, \PGO^+_4)$ or 
$H^1(\Gamma, \Sym_2^3 \rtimes \Sym_4)$. The groups  $\PGO^+_8$ and
$\Sym_2^3 \rtimes \Sym_4$   admit up to inner automorphisms one
 outer action of the cyclic group of order $3$ (this is well known for
 $ \PGO^+_8)$ and is in \cite{franzsen_howlett} for $\Sym_2^3 \rtimes \Sym_4$).
Thus the cyclic group of order $3$ acts on 
 the corresponding sets of isomorphism classes.
In fact these automorphisms of order $3$ are induced by automorphisms that
can already be defined at the level of the groupoids.
Let $\D_4$ be any of the four groupoids above, let $\widetilde \D_4$ be the
corresponding groupoid of unoriented objects, and let
$\mathcal{F}\colon (E,\partial_E) \mapsto  E$ be the forgetful functor
$\D_4  \to \widetilde\D_4$. 

\begin{thm} \label{thm:Candtriality}
  The functors $C_1$, $C_2$ factor through the forgetful functor,
  i.e.,
  there are functors 
  \[
  C_1^+,\; C_2^+\colon  \D_4\to
   \D_4
  \]
  such that $\mathcal{F}\circ C_i^+ = C^i$ for
  $i=1$, $2$. These functors satisfy natural equivalences:
  \[
  (C_1^+)^3 = \id,\qquad (C_1^+)^2=C_2^+
  \]
and induce the trialitarian action on the sets of isomorphism classes.
\end{thm}

\begin{proof}
The claim for thin objects is in~\cite{KT09} (under the hypothesis
that $\operatorname{char}(F)\neq2$) and the classical case
(for algebras) is in~\cite[\S35]{KMRT}.
\end{proof}

Observe that triality is compatible with general embeddings.  Triality
for thin objects is extensively discussed in~\cite{KT09}. A very nice
presentation
of classical (geometric and algebraic) triality is in
\cite{springer_triality}. 

\bibliographystyle{plain}


\end{document}